\documentclass[11pt,a4paper]{article}
\usepackage[T1]{fontenc}
\usepackage{graphicx}
\usepackage{amssymb,amsmath,amsthm}
\usepackage[margin=2.5cm]{geometry}
\usepackage[english]{babel}
\usepackage[latin1]{inputenc}
\usepackage{enumerate}
\usepackage[all]{xy}
\usepackage{multicol}
\usepackage{varioref}

\theoremstyle{plain}
\newtheorem{theo}{Theorem}[section]
\newtheorem{lemma}[theo]{Lemma}
\newtheorem{cor}[theo]{Corollary}
\newtheorem{prop}[theo]{Proposition}

\newtheorem{lemmaa}{Lemma}
\newtheorem{cora}{Corollary}

{}

\theoremstyle{definition}
\newtheorem{defi}[theo]{Definition}

\theoremstyle{remark}

\begin{document}
\title{Quantum groups of $GL(2)$ representation type}
\author{Colin MROZINSKI}
\date{}
\maketitle
\begin{abstract}
We classify the cosemisimple Hopf algebras whose corepresentation semi-ring is isomorphic to that of $GL(2)$. This leads us to define a new family of Hopf algebras which generalize the quantum similitude group of a non-degenerate bilinear form. A detailed study of these Hopf algebras gives us an isomorphic classification and the description of their corepresentation categories.
\end{abstract}
\section{Introduction and main results}

There are many approaches to the classification problem for quantum groups, depending on what group theory aspect one wants to emulate. Our approach is based on Tannaka-Krein reconstruction theory, which shows deep links between a Hopf algebra and its corepresentation category. Keeping that in mind, we investigate the problem of classifying Hopf algebras according to their corepresentation semi-ring, a problem already considered by several authors \cite{Wo,WZ,KP,Ban1,Ban,Ohn2,Ohn,PHH,Bi1}. In the present paper, we consider the $GL(2)$-case and we classify (in characteristic zero) the cosemisimple Hopf algebras having a corepresentation semi-ring isomorphic to the one of $GL(2)$.

Let $k$ be an algebraically closed field, let $n\in \mathbb N, n\geq 2$ and let $A,B\in GL_n(k)$.
We consider the following algebra $\mathcal{G}(A,B)$: it is the universal algebra with generators $(x_{ij})_{1\leq i,j \leq n}, d, d^{-1}$ satisfying the relations 
$$x^tAx=Ad \ \ \ xBx^t=Bd \ \ \ dd^{-1}=1=d^{-1}d,$$
where $x$ is the matrix $(x_{ij})_{1\leq i,j \leq n}$.
This algebra has a natural Hopf algebra structure and might be seen as a generalization
of the Hopf algebra corresponding to the quantum similitude group of a non-degenerate bilinear form.
When $n=2$ and for particular matrices $A,B$, it was used by Ohn (\cite{Ohn})  in order to classify quantum $GL_2(\Bbb C)$'s. 
Let $q\in k^*$. For a well chosen matrix $A_q\in GL_2(k)$, we have $\mathcal{G}(A_q,A_q)=\mathcal{O}(GL_q(2))$, the function algebra on the quantum group $GL_q(2)$. Our first result describes the monoidal category of comodules over $\mathcal{G}(A,B)$ for some matrices $A,B\in GL_n(k)$.

\begin{theo}\label{T 1}
Let $A,B\in GL_n(k)$ ($n\geq 2$) such that $B^tA^tBA=\lambda I_n$ for some $\lambda \in k^*$
and let  $q\in k^*$ such that $ q^2-\sqrt{\lambda^{-1}}tr(AB^t) q+1=0$. Then there is a $k$-linear equivalence of monoidal categories $$\emph{\textrm{Comod}}(\mathcal{G}(A,B))\simeq^{\otimes}\emph{\textrm{Comod}}(\mathcal{O}(GL_q(2)))$$ between the comodule categories of $\mathcal{G}(A,B)$ and $\mathcal{O}(GL_q(2))$ respectively.
\end{theo}

This result is inspired by the paper of Bichon \cite{Bi1}, which gives similar results for the quantum group of a non-degenerate bilinear form. As in \cite{Bi1}, the result is proved
by constructing some appropriate Hopf bi-Galois objects and by using a theorem of Schauenburg \cite{Sch1}.
The Hopf bi-Galois objects we construct are part of a connected cogroupoid \cite{Bi2}.  
The technical difficulty in this approach is to study the connectedness of this cogroupoid.


We use Theorem \ref{T 1} to classify, in characteristic zero, all the cosemisimple Hopf algebras whose corepresentation semi-ring is isomorphic to that of $GL_2(k)$. Recall that $q \in k^*$ is said to be generic if $q$ is not a root of unity or if $q\in \{\pm 1 \}$.

\begin{theo}\label{T 2}
Assume that $char(k)=0$. The Hopf algebras whose corepresentation semi-ring is isomorphic to that of $GL_2(k)$ are exactly the $$\mathcal{G}(A,B)$$ with $A,B\in GL_n(k)$ ($n\geq 2$) satisfying $B^tA^tBA=\lambda I_n$ for some $\lambda \in k^*$ and such that any solution of the equation $ X^2-\sqrt{\lambda^{-1}}tr(AB^t) X+1=0$ is generic.
\end{theo}

A particular case of the theorem was already known if one requires the fundamental comodule of $H$ to be of dimension 2 (\cite{Ohn}). 
A similar classification (without dimension constraint) was obtained by Bichon (\cite{Bi1}) in the  $SL(2)$ case
(the compact $SU(2)$ case had been done by Banica \cite{Ban1}).  The $SL(3)$ case with dimension constraints has been studied by Ohn (\cite{Ohn2}). Other related results have been given in the $SU(N)$ and $SL(N)$ case by Banica (\cite{Ban}) and Phung Ho Hai (\cite{PHH}), in terms of Hecke symmetries. It is worth to note that in principle Theorem \ref{T 2} could be deduced by the combination
of Phung Ho Hai's work \cite{PHH} and Gurevich's classification of Hecke symmetries of rank two \cite{Gu}.
We believe that the present approach, using directly pairs of invertible matrices, is more explicit and simpler.

We also give a version of Theorem \ref{T 2} in the compact case.

Finally the following theorem will complete the classification of $GL(2)$-deformations.
\begin{theo}\label{T 3}
Assume that $char(k)=0$. Let $A,B\in GL_n(k)$ and let $C,D\in GL_m(k)$ such that $B^tA^tBA=\lambda_1 I_n$ and $D^tC^tDC=\lambda_2 I_m$ for $\lambda_1, \lambda_2 \in k^*$. The Hopf algebras $\mathcal{G}(A,B)$ and $\mathcal{G}(C,D)$ are isomorphic if and only if $n=m$ and there exists $P\in GL_n(k)$ such that either $$(C,D)=(P^tAP,  P^{-1}BP^{-1t}) \text{ or } (C,D)=( P^tB^{-1}P, P^{-1}A^{-1}P^{-1t})$$
\end{theo}

We will also provide the classification of the Hopf algebra $\mathcal{G}(A,B)$ up to monoidal equivalence (Corollary \ref{Cor1}).

The paper is organized as follows: in Sec. 2 we introduce the Hopf algebras $\mathcal{G}(A,B)$  and discuss some basic properties; in Sec. 3, we build a cogroupoid linking the Hopf algebra $\mathcal{G}(A,B)$ and study its connectedness: this will prove Theorem \ref{T 1}; in Sec. 4 we prove Theorem $\ref{T 2}$ and Theorem \ref{T 3}; in Sec. 5, we classify $\mathcal{G}(A,B)$-Galois objects up to isomorphisms, its group of bi-Galois objects and its lazy cohomology group; finally, Sec. 6 is devoted to study the $GL(2)$-deformations in the compact case.

\smallskip

Throughout the paper $k$ is an algebraically closed field. 
We assume that the reader is familiar with Hopf algebras and their monoidal categories of comodules (corepresentations), and with 
Hopf-Galois objects. See \cite{mon, scsurv}.

\section{The Hopf algebra $\mathcal{G}(A,B)$}

Let $n\geq 2$ and $A,B \in GL_n(k)$.
The algebra $\mathcal{G}(A,B)$ has been defined in the introduction.
In this section, we briefly discuss its Hopf algebra structure, its universal property
and some of its basic properties. 

The following result will be generalized at the cogroupoid level in the next section, where the proof 
is given.

\begin{prop}
The algebra $\mathcal{G}(A,B)$ admits a Hopf algebra structure, with comultiplication $\Delta$ defined by $$\Delta(x_{ij})=\sum_{k=1}^n x_{ik}\otimes x_{kj},1\leq i,j\leq n, \ \Delta (d^{\pm})=d^{\pm}\otimes d^{\pm},$$ with counit $\varepsilon$ defined by $$\varepsilon(x_{ij})=\delta_{ij}, 1\leq i,j \leq n, \ \varepsilon (d^{\pm})=1$$ and with antipode $S$ defined by $$S(x)=d^{-1}A^{-1}x^tA, \ S(d^{\pm})=d^{\mp}$$
\end{prop}


We now give (and sketch the proof of) the universal property of the Hopf algebra $\mathcal{G}(A,B)$: 
\begin{prop}
Let $H$ be a Hopf algebra with a group-like element $d\in Gr(H)$ and let $V$ be a finite-dimensional $H$-comodule of dimension $n$. Let $a : V\otimes V \to D$ and $b : D \to V\otimes V$ be two $H$-comodule morphisms (where $D$ denotes the $H$-comodule induced by $d$) such that the underlying bilinear forms are non-degenerate. Then there exist $A,B\in GL_n(k)$ such that: 
\begin{enumerate}
\item $V$ and $D$ have a $\mathcal{G}(A,B)$-comodule structure and $a$ and $b$ are $\mathcal{G}(A,B)$-comodule morphisms,
\item there exists a unique Hopf algebra morphism $\psi : \mathcal{G}(A,B) \to H$ such that $(id_D\otimes \psi)\circ \alpha_D=\alpha_D'$ and $(id_V\otimes \psi)\circ \alpha_V=\alpha_V'$ (where $\alpha$ and $\alpha'$ denote the coactions of $\mathcal{G}(A,B)$ and $H$ respectively).
\end{enumerate}
\end{prop}

\begin{proof}
Let $(v_i)_{1\leq i \leq n}$ be a basis of $V$ and $x=(x_{ij})_{1\leq i, j \leq n}$ be the associated matrix of coefficients. Let $A=(a_{ij})_{1\leq i,j \leq n}$, $B=(b_{ij})_{1\leq i, j \leq n}$ be the matrices such that $a(v_i\otimes v_j)=a_{ij}d$ and $b(d)=\sum_{ij} b_{ij}v_i\otimes v_j$. It is straightforward to check that $a$ and $b$ are $H$-colinear if and only if $x^tAx=Ad$ and $xBx^t=Bd$. Finally, since $Gr(H)$ is a group, there exists $d^{-1}\in H$ such that $dd^{-1}=1=d^{-1}d$. The universal property of $\mathcal{G}(A,B)$ gives us the result.
\end{proof}

The following lemma will limit our choice of matrices $A,B \in GL_n(k)$. The proof comes directly from Schur's lemma.
\begin{lemma}\label{L Schur}
Let $H$ be as in the previous proposition, and assume that the $H$-comodule $V$ is irreducible. Then the composition 
$$\begin{array}{c} D\otimes V \stackrel{b \otimes id}{\longrightarrow} V\otimes V\otimes V \stackrel{id\otimes a}{\longrightarrow} V\otimes D \stackrel{id\otimes b}{\longrightarrow} V\otimes V\otimes V \stackrel{a \otimes id}{\longrightarrow} D\otimes V \end{array}$$
is a multiple of the identity, i.e there exists $\lambda \in k^*$ such that :
$$ (a \otimes id) \circ (id\otimes b) \circ(id\otimes a)\circ(b \otimes id)=\lambda id_{D\otimes V}$$ When $H=\mathcal{G}(A,B)$, the relation may be rewritten as $$B^tA^tBA=\lambda I_n$$
\end{lemma}

The next result is part of the isomorphic classification of the Hopf algebras $\mathcal{G}(A,B)$.

\begin{prop}\label{P 2}
Let $A,B\in GL_n(k)$ and let $P,Q\in GL_n(k)$. The Hopf algebras $\mathcal{G}(A,B)$, $\mathcal{G}(P^tAP,$ $ P^{-1}BP^{-1t})$ and $\mathcal{G}( Q^tB^{-1}Q, Q^{-1}A^{-1}Q^{-1t})$ are isomorphic.
\end{prop}
\begin{proof}
Considering the first case, we denote by $x_{ij}$, $d$ and $d^{-1}$, $y_{ij}$, $d$ and $d^{-1}$ ($1\leq i,j \leq n$) the respective generators of $\mathcal{G}(A,B)$ and $\mathcal{G}(P^tAP,$ $ P^{-1}BP^{-1t})$. The defining relations $$ y^t(P^tAP)y=(P^tAP)d,  \text{ and } y(P^{-1}BP^{-1t})y^t=(P^{-1}BP^{-1t})d $$
ensure that we have an isomorphism $$f : \mathcal{G}(A,B) \to \mathcal{G}(P^tAP, P^{-1}BP^{-1t})$$ satisfying $f(x)=PyP^{-1}, f(d)=d$ and $f(d^{-1})=d^{-1}$, with inverse $f^{-1}(y)=P^{-1}xP, f^{-1}(d)=d$ and $f^{-1}(d^{-1})=d^{-1}$.

In the second case, denoting $y_{ij}$ ($1\leq i,j \leq n$), $d$ and $d^{-1}$ the generators of $\mathcal{G}( Q^tB^{-1}Q, Q^{-1}A^{-1}Q^{-1t})$, the same considerations on the defining relations $$ y^t(Q^tB^{-1}Q)y=(Q^tB^{-1}Q)d,  \text{ and } y(Q^{-1}A^{-1}Q^{-1t})y^t=(Q^{-1}A^{-1}Q^{-1t})d $$ together with the commutation relations in $\mathcal{G}(A,B)$ $$(AB)x^td^{-1}=d^{-1}x^t(AB)$$ give us an isomorphism $$f : \mathcal{G}(A,B) \to \mathcal{G}( Q^tB^{-1}Q, Q^{-1}A^{-1}Q^{-1t})$$ satisfying $f(x)=Qyd^{-1}Q^{-1}, f(d)=d^{-1}$ and $f(d^{-1})=d$, with inverse $f^{-1}(y)=Q^{-1}xd^{-1}Q$, $f^{-1}(d)=d^{-1}$ and $f^{-1}(d^{-1})=d$.
\end{proof}

Let us note that for a good choice of matrices $A,B\in GL_n(k)$, the Hopf algebra $\mathcal{G}(A,B)$ coincides with the standard quantization of the function algebra $\mathcal{O}(GL_{2}(k))$: precisely, a straightforward computation shows that 
\begin{itemize}
\item for $A=\left(\begin{array}{cc}
0 & 1 \\
-q & 0 
\end{array}\right):=A_q$ and $B=A_p$, for some $q,p\in k^*$, we get the two-parameter standard quantum $GL_2(k)$: $$\mathcal{G}(A_q,A_p)=\mathcal{O}(GL_{q,p}(2))$$ 
\item and for $A=\left(\begin{array}{cc}
0 & 1 \\
-1 & h 
\end{array}\right)$ and $B=\left(\begin{array}{cc}
-h' & 1 \\
-1 & 0 
\end{array}\right)$, with $h,h' \in k$, we get the Jordanian quantum case: $$\mathcal{G}(A,B)=\mathcal{O}_{h,h'}^{J}(GL(2))$$
\end{itemize}

(the defining relations of this two algebras can be found in \cite{Ohn}).

Moreover, we can see that we have a surjective Hopf algebra morphism $$\mathcal{G}(A,A^{-1})\to \mathcal{B}(A)$$ where $\mathcal{B}(A)$ is the Hopf algebra representing the quantum automorphism group of the non-degenerate bilinear form associated to $A$, introduced by Dubois-Violette and Launer in \cite{DVL}. In view of its definition, we can consider $\mathcal{G}(A,A^{-1})$ as the Hopf algebra representing the quantum similitude group of this non degenerate bilinear form.

\section{The cogroupoid $\mathcal{G}$}

To prove Theorem \ref{T 1} by using Schauenburg's results from \cite{Sch1} , we now proceed to construct Hopf-bigalois objects linking the Hopf algebras $\mathcal G(A,B)$ and in order to do our computations in a nice context, we put the algebras $\mathcal{G}(A,B)$ in a cogroupoid framework. We recall some basic definitions and facts about these objects (for more precise informations, we refer to \cite{Bi2}).

\begin{defi}
A $k$-\emph{cogroupoid} $C$ consists of:
\begin{itemize}
 \item a set of objects $ob(C)$.
 \item For any $X,Y \in ob(C)$, a $k$-algebra $C(X,Y)$.
 \item For any $X,Y,Z \in ob(C)$, algebra morphisms $$ \Delta_{X,Y}^Z : C(X,Y) \to C(X,Z)\otimes C(Z,Y) \text{ and } \ \varepsilon_X : C(X,X)\to k$$ and linear maps  $$S_{X,Y} : C(X,Y) \to C(Y,X)$$ satisfying several compatibility diagrams: see \cite{Bi2}, the axioms are dual to the axioms defining a groupoid.
\end{itemize}

A cogroupoid $C$ is said to be \emph{connected} if $C(X,Y)$ is a non zero algebra for any $X,Y \in ob(C)$.
\end{defi}

Let $n,m\in \Bbb N, n,m \geq 2$ and let $A,B\in GL_n(k)$, $C,D\in GL_m(k)$. We define the algebra $$\mathcal{G}(A,B|C,D):=k\left<d, d^{-1}, x_{i,j}, 1\leq i \leq n, 1\leq j \leq m \bigg| 
\begin{aligned}
&x^tAx=Cd, \\
&xDx^t=Bd
\end{aligned}, \ d^{-1}d=1=dd^{-1}\right>
$$ 

Of course the generators $x_{ij}, d $ and $d^{-1}$ in $\mathcal{G}(A,B|C,D)$ should be denoted by $x_{i,j}^{AB,CD}$, $d_{AB,CD}$  and $d^{-1}_{AB,CD}$ to express the dependence on $(A,B), (C,D)$, but there will be no confusion and we simply denote them by $x_{ij}, d $ and $d^{-1}$. It is clear that $\mathcal{G}(A,B|A,B)=\mathcal{G}(A,B)$.

In the following lemma, we construct the structural maps that will put the algebras $\mathcal{G}(A,B|C,D)$ in a cogroupoid framework.

\begin{lemma}\label{L 4}
\begin{itemize}

\item For any $A,B\in GL_n(k)$, $C,D \in GL_m(k)$ and $X,Y \in GL_p(k)$, there exist algebra maps 
$$\Delta_{AB,CD}^{XY}:\mathcal{G}(A,B|C,D)\to \mathcal{G}(A,B|X,Y)\otimes\mathcal{G}(X,Y|C,D)$$
such that $\Delta(x_{ij}) = \sum_{k=1}^p x_{ik}\otimes x_{kj}$ ($1\leq i \leq n, 1\leq j \leq m$), $\Delta(d^{-1})=d^{-1}\otimes d^{-1}$, and
$$\varepsilon_{AB}: \mathcal{G}(A,B)\to k$$ such that $\varepsilon_{AB} (x_{ij})= \delta_{ij}$ ($1\leq i \leq n, 1\leq j \leq m$), $\varepsilon (d)=1=\varepsilon (d^{-1})$, and for any $M,N\in GL_r(k)$, the following diagrams commute:

$$\xymatrix{
\mathcal{G}(A,B|C,D) \ar[r]^{\Delta_{AB,CD}^{XY} \ \ \ \ \ \ } \ar[d]_{\Delta_{AB,CD}^{MN}} & \mathcal{G}(A,B|X,Y)\otimes \mathcal{G}(X,Y|C,D) \ar[d]^{\Delta_{AB,XY}^{MN}\otimes id} \\
\mathcal{G}(A,B|M,N)\otimes \mathcal{G}(M,N|C,D) \ar[r]_{id\otimes \Delta_{MN,CD}^{XY} \ \ \ \ \ \ } & \mathcal{G}(A,B|M,N)\otimes \mathcal{G}(M,N|X,Y)\otimes \mathcal{G}(X,Y|C,D)
}$$

$$\xymatrix{
\mathcal{G}(A,B|C,D)  \ar[d]_{\Delta_{AB,CD}^{CD}} \ar[rd]& \\
\mathcal{G}(A,B|C,D)\otimes \mathcal{G}(C,D) \ar[r]_{\ \ \ \ \ \ id\otimes \varepsilon_{CD}} & \mathcal{G}(A,B|C,D)
}
\xymatrix{
\mathcal{G}(A,B|C,D)  \ar[d]_{\Delta_{AB,CD}^{AB}} \ar[rd]& \\
\mathcal{G}(A,B)\otimes \mathcal{G}(A,B|C,D) \ar[r]_{\ \ \ \ \ \ \varepsilon_{AB}\otimes id} & \mathcal{G}(A,B|C,D)
}$$

\item For any $A,B\in GL_n(k)$, $C,D\in GL_m(k)$, there exists an algebra map 
$$ S_{AB,CD}:\mathcal{G}(A,B|C,D) \to \mathcal{G}(C,D|A,B)^{op}$$

defined by the formula $S_{AB,CD}(x)=A^{-1}d^{-1}x^tC$, $S_{AB,CD}(d^{\pm 1})=d^{\mp 1}$,
such that the following diagrams commute: 
$$\xymatrix{
\mathcal{G}(A,B) \ar[r]^{\varepsilon_{AB}} \ar[d]_{\Delta_{AB,AB}^{CD}}& k \ar[r]^{u \ \ \ \ \ }& \mathcal{G}(A,B|C,D)  \\
\mathcal{G}(A,B|C,D)\otimes\mathcal{G}(C,D|A,B) \ar[rr]_{id\otimes S_{CD,AB}} & & \mathcal{G}(A,B|C,D)\otimes\mathcal{G}(A,B|C,D)\ar[u]^{m} 
}$$ $$
\xymatrix{
\mathcal{G}(A,B) \ar[r]^{\varepsilon_{AB}} \ar[d]_{\Delta_{AB,AB}^{CD}}& k \ar[r]^{u \ \ \ \ \ }& \mathcal{G}(A,B|C,D)  \\
\mathcal{G}(A,B|C,D)\otimes\mathcal{G}(C,D|A,B) \ar[rr]_{S_{AB,CD}\otimes id} & & \mathcal{G}(C,D|A,B)\otimes\mathcal{G}(A,B|C,D)\ar[u]^{m} 
}$$
\end{itemize}
\end{lemma}
\begin{proof}
First we have to check that the algebra maps are well defined. 

Let $A,B\in GL_n(k)$, $C,D \in GL_m(k)$ and $X,Y \in GL_p(k)$; in order to simplify the notations, we denote $\Delta_{AB,CD}^{XY}=\Delta$, $\varepsilon_{AB}=\varepsilon$ and $S_{AB,CD}=S$. We only give the computations for the first relation $x^tAx=Cd$, the computations for second one being similar. 

For $\Delta : \mathcal{G}(A,B|C,D)\to \mathcal{G}(A,B|X,Y)\otimes\mathcal{G}(X,Y|C,D)$, we compute: 
$$\begin{aligned}
\Delta(x^tAx)_{ij}&=\Delta (\sum_{kl}A_{kl}x_{ki}x_{lj})=\sum_{kl}A_{kl}(\sum_p x_{kp}\otimes x_{pi})(\sum_q x_{lq}\otimes x_{qj})\\
&=\sum_{pq}\sum_{kl} A_{kl}x_{kp}x_{lq}\otimes x_{pi}x_{qj}=\sum_{pq} (x^tAx)_{pq}\otimes x_{pi}x_{qj}\in \mathcal{G}(A,B|X,Y)\otimes\mathcal{G}(X,Y|C,D)\\
&=\sum_{pq} X_{pq}d\otimes x_{pi}x_{qj}=d\otimes(x^tXx)_{ij}=C_{ij} d\otimes d
\end{aligned}$$

and the computations for $\varepsilon: \mathcal{G}(A,B)\to k$ are:
$$\begin{aligned}
\varepsilon ((x^tAx)_{ij})&=\varepsilon(\sum_{kl} A_{kl}x_{ki}x_{lj})
=\sum_{kl} A_{kl} \varepsilon (x_{ki})\varepsilon (x_{lj})\\
&=\sum_{kl} A_{kl} \delta_{ki}\delta_{lj}=A_{ij}=A_{ij}\varepsilon(d)
\end{aligned}$$

Then $\Delta_{AB,CD}^{XY}$ and $\varepsilon_{AB}$ are well defined. These maps are algebra maps, so it is enough to check the commutativity of the diagrams of the first part on the generators of $\mathcal{G}(A,B|C,D)$, which is obvious. 

Recall that if $\Phi : A \to B^{op}$ is an algebra map, then we have $\Phi(ab)=(\Phi(b)^t\Phi(a)^t)^t$ for all matrices $a,b\in M_n(A)$.

Then, for $S:\mathcal{G}(A,B|C,D) \to \mathcal{G}(C,D|A,B)^{op}$, we have:

$$\begin{aligned}
 S(C^{-1}d^{-1}x^tAx)&=\big(S(x)^tA^tS(x)d(C^{-1})^t\big)^t\\
 &=\big((C^td^{-1}x(A^{-1})^t)A^t(d^{-1}A^{-1}x^tC)d(C^{-1})^t\big)^t\\
 &=\big(C^td^{-1}xd^{-1}A^{-1}x^tCd(C^{-1})^t\big)^t\\
 &=\big(C^td^{-1}d(C^{-1})^t\big)^t\\
 &=1
\end{aligned}$$

We can check in the same way that $S$ is compatible with the second relation, and then $S=S_{AB,CD}$ is well defined.
The commutativity of the diagrams follows from the verification on the generators of $\mathcal{G}(A,B)$ and the fact that $\Delta_{\bullet,\bullet}^{\bullet}$, $\varepsilon_{\bullet}$ and $S_{\bullet,\bullet}$ are algebra maps.
\end{proof}

The lemma allows the following definition:

\begin{defi}
The cogroupoid $\mathcal{G}$ is the cogroupoid defined as follows:
\begin{enumerate}[(i)]
\item $ob(\mathcal{G})=\{(A,B)\in GL_m(k)\times GL_m(k), m\geq 1\}$,
\item For $(A,B),(C,D)\in ob(\mathcal{G})$, the algebra $\mathcal{G}(A,B|C,D)$ is the algebra defined above,
\item the structural maps $\Delta_{\bullet,\bullet}^{\bullet}$, $\varepsilon_{\bullet}$ and $S_{\bullet,\bullet}$ are defined in the previous lemma.
\end{enumerate}
\end{defi}

So we have a cogroupoid linking all the Hopf algebras $\mathcal{G}(A,B)$. 
The following result is part of the isomorphic classification of the algebras $\mathcal{G}(A,B|C,D)$, which will be completed in Theorem \ref{T 5}, and this proposition will be used in the appendix.

\begin{prop}\label{P 1}
Let $A,B,P \in GL_n(k)$, $C,D, Q \in GL_m(k)$. We have algebra isomorphisms
$$\mathcal{G}(A,B|C,D) \simeq \mathcal{G}(P^tAP,P^{-1}BP^{-1t}|Q^tCQ,Q^{-1}DQ^{-1t}) $$
$$ \mathcal{G}(A,B|C,D) \simeq \mathcal{G}(B^{-1},A^{-1}|D^{-1},C^{-1})$$ 
\end{prop}

\begin{proof}
For the first case, let us denote the generators of $\mathcal{G}(P^tAP,P^{-1}BP^{-1t}|Q^tCQ,Q^{-1}DQ^{-1t})$ by $y_{ij}$ ($1\leq i \leq n, 1\leq j\leq m$), $d,d^{-1}$.
Then the relations $$ x^t(P^tAP)x=(Q^tCQ)d,  \text{ and } x(Q^{-1}DQ^{-1t})x^t=(P^{-1}BP^{-1t})d   $$ ensure that we have an algebra morphism $\psi : \mathcal{G}(A,B|C,D) \to \mathcal{G}(P^tAP,P^{-1}BP^{-1t}|Q^tCQ,Q^{-1}DQ^{-1t})$ defined by $\psi(d)=d$, $\psi(d^{-1})=d^{-1}$ and $\psi(x)=PyQ^{-1}$. The inverse map is then defined by $\psi^{-1}(d)=d$, $\psi^{-1}(d^{-1})=d^{-1}$ et $\psi^{-1}(y)=P^{-1}xQ$

For the second case, let us denote the generators of $\mathcal{G}(B^{-1},A^{-1}|D^{-1},C^{-1})$ by $y_{ij}$ ($1\leq i \leq n, 1\leq j\leq m$), $d,d^{-1}$. Then the relations $$y^tB^{-1}y=D^{-1}d, \text{ and } yC^{-1}y^t=A^{-1}d$$ ensure that we have an algebra morphism $\psi : \mathcal{G}(A,B|C,D) \to \mathcal{G}(B^{-1},A^{-1}|D^{-1},C^{-1})$ given by $\psi(d^{\pm})=d^{\mp}$ and $\psi(x)=yd^{-1}$. This is an isomorphism with inverse map defined by $\psi^{-1}(d^{\pm})=d^{\mp}$ and $\psi^{-1}(y)=xd^{-1}$.
\end{proof}

Now the natural question is to study the connectedness of $\mathcal{G}$, which will ensure that we indeed get
Hopf-Galois objects and hence equivalences of monoidal categories.

\begin{lemma}\label{L 1}
Let $q \in k^*$ and let $C,D\in GL_m(k)$ such that $tr(CD^t)=1+q^2$ and $D^tC^tDC=q^2I_m$. Then the algebra $\mathcal{G}(A_q,A_q|C,D)$ is non zero.
\end{lemma}

The (technical) proof of this result is done in the appendix. 
We get the following corollary.

\begin{cor}
Let $\lambda, \mu \in k^*$. Consider the full subcogroupoid $\mathcal{G}^{\lambda, \mu}$ of $\mathcal{G}$ with objects $$ob(\mathcal{G}^{\lambda, \mu})=\{(A,B)\in ob(\mathcal{G}) \ ; \ B^tA^tBA=\lambda I_n  \text{ and } tr(AB^t)=\mu   \}$$
Then $\mathcal{G}^{\lambda, \mu}$ is a connected cogroupoid.
\end{cor}

\begin{proof}
Let $(A,B) \in ob(\mathcal{G}^{\lambda, \mu})$. By the relations defining those algebras, if $\alpha, \beta \in k^*$, $C,D\in GL_m(k)$ then: $$\mathcal{G}(A,B|C,D)=\mathcal{G}(\alpha A, \beta B|\alpha C,\beta D)$$
Choose $q\in k^*$ satisfying $q^2-\sqrt{\lambda^{-1}}\mu q+1=0$ and put $A'=\sqrt{\lambda^{-1}}A$ and $B'=qB$. We have  $tr(A'B'^t)=1+q^2$ and $B'^tA'^tB'A'=q^2I_m$. By Lemma \ref{L 1}, we have that $\mathcal{G}(A_q,A_q|A',B')$ is non zero and so is $\mathcal{G}(\sqrt{\lambda}A_q,q^{-1}A_q|A,B)$. Then we have found $X\in ob(\mathcal{G}^{\lambda, \mu})$ such that $\mathcal{G}(X|A,B)\neq (0)$ for all $(A,B)\in ob(\mathcal{G}^{\lambda, \mu})$. According to \cite{Bi2}, Proposition 2.15, the cogroupoid $\mathcal{G}^{\lambda, \mu}$ is connected.  
\end{proof}

Hence by \cite{Bi2}, Proposition 2.8 and Schauenburg's Theorem 5.5 \cite{Sch1}, we have the following result:

\begin{theo}
Let $(A,B), (C,D)\in ob(\mathcal{G}^{\lambda, \mu})$. Then we have a $k$-linear equivalence of  monoidal categories $$\emph{\textrm{Comod}}(\mathcal{G}(A,B))\simeq^{\otimes}\emph{\textrm{Comod}}(\mathcal{G}(C,D))$$ between the comodule categories of $\mathcal{G}(A,B)$ and $\mathcal{G}(C,D)$ respectively.
\end{theo}

We are ready to prove Theorem \ref{T 1}.

\begin{proof}[Proof of Theorem \ref{T 1}]
First, note that we have $\mathcal{G}(A,B)=\mathcal{G}(\alpha A,\beta B)$ for all $\alpha, \beta \in k^*$. Let $q\in k^*$ such that $q^2-\sqrt{\lambda^{-1}}tr(AB^t)q+1=0$. Then, by the above theorem, we have a $k$-linear equivalence of monoidal categories $$\textrm{Comod}(\mathcal{G}(A,B))=\textrm{Comod}(\mathcal{G}(\sqrt{\lambda^{-1}}A,qB))\simeq^{\otimes}\textrm{Comod}(\mathcal{O}(GL_q(2)))$$
and we are done.
\end{proof}


%

\section{$GL(2)$-deformations}

In this section $k$ will be an algebraically closed field of characteristic zero. This paragraph is essentially devoted to the proof of Theorem \ref{T 2}. We also complete the isomorphic and Morita equivalence classifications of the Hopf algebras $\mathcal{G}(A,B)$.

Recall that the corepresentation semi-ring (or fusion semi-ring) of a cosemisimple Hopf algebra $H$, denoted $\mathcal R^+(H)$,
is the set of isomorphism classes of finite-dimensional $H$-comodules. 
The direct sum of comodules defines the addition while the tensor product of comodules defines
the multiplication. The isomorphism classes of simple $H$-comodules form a basis of $\mathcal R^+(H)$. 
The isomorphism class of a finite-dimensional $H$-comodule $V$ is denoted by $[V]$. 

Let $K$ be another cosemisimple Hopf algebra, and let $f : H \to K$ a Hopf algebra morphism. Then $f$ induces a monoidal functor $f_*: \textrm{Comod}_f(H) \to \textrm{Comod}_f(K)$ and a semi-ring morphism $f_* :\mathcal{R}^+(H) \to\mathcal{R}^+(K)$. A semi-ring isomorphism $\mathcal{R}^+(H) \simeq \mathcal{R}^+(K)$ induces a bijective correspondence (that preserves tensor products) between the isomorphism classes of simple comodules of $H$ and $K$.

Let $G$ be a  reductive algebraic group. As usual we say that the cosemisimple Hopf algebra $H$ is a $G$-deformation if $\mathcal{R}^+(\mathcal{O}(G)) \simeq \mathcal{R}^+(H)$. Hence Theorem \ref{T 2} classifies $GL(2)$-deformations.\\

We now recall the representation theory of $GL_q(2)$. Our references are Ohn \cite{Ohn} for the generic case and the root of unity case can be adapted from the representation theory of $SL_q(2)$ given by Kondratowicz and Podl\`es in \cite{KP}.

\begin{itemize}
\item Let first assume that $q\in k^*$ is generic. Then $\mathcal{O}(GL_q(2))$ is cosemisimple and there are two family $(U_{n})_{n\in \Bbb N }$ and $(D^{\otimes e})_{e\in \Bbb Z }$ of non-isomorphic simple comodules (except for $U_0=D^{\otimes 0}=k$) such that ($(n,e), (m,f) \in \mathbb{N}^*\times \Bbb Z$)
\begin{enumerate}[]
 \item $\dim_k(U_{n})=n+1 \text{ and } \dim_k (D)=1$
 \item $(U_{n}\otimes D^{\otimes e}) \otimes (U_{m}\otimes D^{\otimes f}) \cong (U_{m}\otimes D^{\otimes f}) \otimes (U_{n}\otimes D^{\otimes e}) \cong  \underset{i=0}{\overset{min (n,m)}{\bigoplus}} U_{n+m-2i}\otimes D^{\otimes e+f+i}$
\end{enumerate}

Moreover, every simple $\mathcal{O}(GL_q(2))$-comodule is isomorphic to one of the comodules $U_{n}\otimes D^{\otimes e}=:U_{(n,e)}$.
\item Now assume that $q\in k^*$ is not generic. Let $N \geq 3$ be its order. Put $$N_0=\left\{ \begin{array}{ll} 
N \ &\text{ if $N$ is odd}\\
N/2 \ &\text{ if $N$ is even}\\
                                       \end{array} \right.$$
Then there exists three families $(V_n)_{n \in \Bbb N}$, $(U_m)_{1\leq m \leq N_0-1}$ and $(D^{\otimes e})_{e\in \Bbb Z }$ of non-isomorphic simple comodules (except for $V_0=U_0=D^{\otimes 0}=k$) such that ($n\in \Bbb N, m=0,1,\dots, N_0-1$)
\begin{enumerate}[]
\item $\dim_k(V_n)=n+1, \dim_k(U_m)=m+1 \text{ and } \dim(D)=1$,
\item $V_n\otimes V_1 \simeq V_1\otimes V_n \simeq V_{n+1} \oplus \big(V_{n-1}\otimes D^{\otimes N_0}\big)$,
\item $U_m\otimes U_1 \simeq U_1\otimes U_m \simeq U_{m+1} \oplus \big(U_{m-1}\otimes D\big)$
\end{enumerate}

Moreover the comodules $V_n\otimes U_m\otimes D^{\otimes e}$ are simple and every simple $\mathcal{O}(GL_q(2))$-comodule is isomorphic to one of these.

The comodule $U_{N_0-1}\otimes U_1$ is not semisimple. It has a simple filtration $$(0)\subset U_{N_0-2}\otimes D \subset Y \subset U_{N_0-1}\otimes U_1$$ such that $$U_{N_0-1}\otimes U_1/Y\simeq U_{N_0-2}\otimes D \text{  and  } Y/U_{N_0-2}\otimes D \simeq V_1$$

\end{itemize}

Let $A,B \in GL_n(k)$. We denote by $V_n^{AB}$, $U_m^{AB}$ and $D_{AB}$ the simple $\mathcal{G}(A,B)$-comodules corresponding to the simple $\mathcal{O}(GL_q(2))$-comodules $V_n$, $U_m$ and $D$, and sometimes we note $U_{(m,e)}^{AB}= U_m^{AB}\otimes D_{AB}^{\otimes e}$.

The lowing lemma will be very useful.

\begin{lemma}\label{L 2}
Let $A,B\in GL_n(k)$ and let $C,D\in GL_m(k)$ such that $B^tA^tBA=\lambda I_n$, $D^tC^tDC=\lambda I_m$ and $tr(AB^t)=tr(CD^t)$. Let $\Omega : \textrm{Comod}(\mathcal{G}(A,B))\to \textrm{Comod}(\mathcal{G}(C,D))$ be an equivalence of monoidal categories. 

If $\mathcal{G}(A,B)$ et $\mathcal{G}(C,D)$ are cosemisimple, we have either, for $(n,e) \in \Bbb N\times \Bbb Z$:
\begin{itemize}
\item  $$\Omega(U_{(0,1)}^{AB})\simeq U_{(0,1)}^{CD} \ \text{and then} \ \Omega(U_{(n,e)}^{AB})\simeq U_{(n,e)}^{CD} $$
\item or $$ \Omega(U_{(0,1)}^{AB})\simeq U_{(0,-1)}^{CD} \ \text{and then} \ \Omega(U_{(n,e)}^{AB})\simeq U_{(n,-n-e)}^{CD} $$
\end{itemize}

If $\mathcal{G}(A,B)$ et $\mathcal{G}(C,D)$ are not cosemisimple, we have either, for $n\in \Bbb N, e \in \Bbb Z, m \in \{0,\dots, N_0-1\}$:
\begin{itemize}
\item  $\Omega(U_{(0,1)}^{AB})\simeq U_{(0,1)}^{CD}$ and then $$\Omega(V_{n}^{AB})\simeq V_{n}^{CD}   \text{and} \ \Omega(U_{(m,e)}^{AB})\simeq U_{(m,e)}^{CD}$$
\item or $ \Omega(U_{(0,1)}^{AB})\simeq U_{(0,-1)}^{CD}$ and then $$ \Omega(V_{n}^{AB})\simeq V_{n}^{CD}\otimes U_{(0,-n)}^{CD}  \text{ and }  \Omega(U_{(m,e)}^{AB})\simeq U_{(m,-m-e)}^{CD} $$
\end{itemize}
\end{lemma}

\begin{proof}
Assume first that the algebras $\mathcal{G}(A,B)$ and $\mathcal{G}(C,D)$ are cosemisimple.
According to the fusion rule, $U_{(0,1)}\otimes U_{(0,1)}\simeq U_{(0,2)}$, so $\Omega(U_{(0,1)}^{AB})\otimes \Omega(U_{(0,1)}^{AB})$ is simple, i.e. there exists $s(\Omega)\in \Bbb Z$ such that $$\Omega(U_{(0,1)}^{AB})\simeq U_{(0,s(\Omega))}^{CD}$$
Then, we have $$\Omega(U_{(0,e)}^{AB})\simeq U_{(0,s(\Omega)e)}^{CD} \ \forall e\in \Bbb Z$$
Similarly, if $\Lambda : \textrm{Comod}(\mathcal{G}(C,D))\to \textrm{Comod}(\mathcal{G}(A,B))$ is a monoidal quasi-inverse for $\Omega$, we have $$\Lambda(U_{(0,e)}^{CD})\simeq U_{(0,s(\Lambda)e)}^{AB} \ \forall e\in \Bbb Z$$ and in particular $$\Lambda(U_{(0,s(\Omega)}^{CD}))\simeq U_{(0,s(\Lambda)s(\Omega))}^{AB}\simeq U_{(0,1)}^{AB}$$ and then $s(\Omega)\in \{-1, 1\}$.

Next, we have $U_{(1,0)}\otimes U_{(1,0)}\simeq U_{(2,0)}\oplus U_{(0,1)}$, and, by the fusion rules, the only simple $\mathcal{G}(C,D)$-comodules $W$ such that $W\otimes W$ is direct sum of two simple comodules are the $(U_{(1,p)}^{CD})_{p\in \Bbb Z}$. Hence there exists $p\in \Bbb Z$ such that $$\Omega(U_{(1,0)}^{AB})\simeq U_{(1,p)}^{CD}$$ We have: by $U_{(1,0)}\otimes U_{(1,0)}\simeq U_{(2,0)}\oplus U_{(0,1)}$ and since $\Omega$ is monoidal we deduce that $$U_{(1,p)}^{CD}\otimes U_{(1,p)}^{CD} \simeq \Omega(U_{(2,0)}^{AB})\oplus U_{(0,s(\Omega))}^{CD}$$ On the other hand: $$U_{(1,p)}^{CD}\otimes U_{(1,p)}^{CD} \simeq U_{(2,2p)}^{CD}\oplus U_{(0,2p+1)}^{CD}$$
We deduce from the uniqueness of the decomposition into simple comodules that $U_{(0,2p+1)}^{CD}\simeq U_{(0,s(\Omega))}^{CD}$ and then that $$(p,s(\Omega))\in \{(0,1), \ (-1,-1)\}$$

By induction, we get $\forall (n,e) \in \Bbb N\times \Bbb Z$, $\Omega(U_{(n,e)}^{AB})\simeq U_{(n,e)}^{CD}$ if $s(\Omega)=1$ and $\Omega(U_{(n,e)}^{AB})\simeq U_{(n,-n-e)}^{CD}$ if $s(\Omega)=-1$

Consider now the non-cosemisimple case:
in the same way as above, we get $$\Omega(U_{(0,1)}^{AB})\simeq U_{(0,s(\Omega))}^{CD} \text{  with  } s(\Omega)\in \{-1, 1\}$$ Moreover $V^{AB}_1$ is simple, so $\Omega(V^{AB}_1)\simeq V^{CD}_n\otimes U^{CD}_{(m,p)}$ with $n\in \Bbb N, m \in \{0,\dots,N_0-1\}$, $(n,m)\neq (0,0), p\in \Bbb Z$. Similarly we have $\Omega(U^{AB}_{(1,0)})\simeq V^{CD}_k\otimes U^{CD}_{(l,t)}$ with $k\in \Bbb N, l \in \{0,\dots,N_0-1\}$, $(k,l)\neq (0,0), t\in \Bbb Z$. Then we have $$\Omega(V^{AB}_1\otimes U^{AB}_{(1,0)})\simeq V^{CD}_n\otimes U^{CD}_{(m,p)}\otimes V^{CD}_k\otimes U^{CD}_{(l,t)}$$ but since $V^{AB}_1\otimes U^{AB}_{(1,0)}$ is still simple, we must have either $n=l=0$ or $m=k=0$. In the first case, we have $\Omega(U^{AB}_{(1,0)})\simeq V^{CD}_k\otimes U^{CD}_{(0,t)}$. But $(U^{AB}_{(1,0)})^{\otimes N_0}$ is not semisimple whereas $(V^{CD}_k)^{\otimes N_0}\otimes U^{CD}_{(0,t+N_0)}$ is. So we have $m=k=0$ and $$\Omega(V^{AB}_1)\simeq V^{CD}_n \otimes U^{CD}_{(0,p)}$$ By the cosemisimple case, we have $$(p,s(\Omega))\in \{(0,1), \ (-1,-1)\}$$ $\Omega(V_n^{AB})\simeq V_n^{CD}$ or $\Omega(V_n^{AB})\simeq V_n^{CD}\otimes U_{(0,-n)}^{CD}, \forall n\in \Bbb N$

Let $Z$ be a simple $\mathcal{G}(A,B)$-comodule such that $\Omega(Z)\simeq U_{(1,0)}^{CD}$. We have $Z\simeq V^{AB}_n\otimes U^{AB}_{(m,e)}$ and then $$U_{(1,0)}^{CD}\simeq V_n^{CD} \otimes \Omega(U^{AB}_{(m,0)}) \otimes U^{CD}_{(0,p)} \ \ (\text{with} \ \  p\in \{e, -n-e\})$$

By the fusion rules we get the following inequalities: $$\dim (\Omega(U^{AB}_{(1,i_1)}))< \dim (\Omega(U^{AB}_{(2,i_2)}))< \dots <\dim (\Omega(U^{AB}_{(N_0-1,i_{N_0-1})}))$$ and then if $m>1$ we have $\dim (\Omega(U^{AB}_{(m,e)})) < \dim (U^{CD}_{(1,j)})$. On the other side, another glance at the fusion rules shows that the $U^{CD}_{(1,j)}, j\in \Bbb Z$, are the simple comodules (that are not one dimensional) of the smallest dimension. Hence $m=1$ and $Z\simeq U^{AB}_{(1,e)}$. The same arguments as above show us that $(e,s(\Omega))\in \{(0,1), \ (-1,-1)\}$.
\end{proof}

We are now able to complete the proof of Theorem \ref{T 3} and the isomorphic classification of the Hopf algebras $\mathcal{G}(A,B)$.

\begin{proof}[Proof of Theorem \ref{T 3}]
We have already proved that the Hopf algebras $\mathcal{G}(A,B)$, $\mathcal{G}(P^tAP,$ $ P^{-1}BP^{-1t})$ and $\mathcal{G}( Q^tB^{-1}Q, Q^{-1}A^{-1}Q^{-1t})$ are isomorphic, see Proposition \ref{P 2}.

To prove the converse, we denote by $x_{ij} \ (1\leq i,j \leq n), d$ and $y_{ij} \ (1\leq i,j \leq m), d$ the respective generators of $\mathcal{G}(A,B)$ and $\mathcal{G}(C,D)$ and by $x$ and $y$ the corresponding matrices. By construction, the elements $(x_{ij})$ and $(y_{ij})$ are the matrix coefficients of the comodules $U_{(1,0)}^{AB}$ and $U_{(1,0)}^{CD}$, and $d$, $d$ those of $U_{(0,1)}^{AB}$ and $U_{(0,1)}^{CD}$.

Let $f : \mathcal{G}(A,B) \to \mathcal{G}(C,D)$ be a Hopf algebra isomorphism and let $f_* : \textrm{Comod}(\mathcal{G}(A,B)) \to \textrm{Comod}(\mathcal{G}(C,D))$ be the induced equivalence of monoidal categories. 

According to lemma \ref{L 2} and its proof, there are two cases:

 If $f_*(U_{(0,1)}^{AB})\simeq U_{(0,1)}^{CD}$ (i.e. if $f(d)=d$) then $f_*(U_{(1,0)}^{AB})\simeq U_{(1,0)}^{CD}$. 

 In this case, $n=m$ and there exists $P\in GL_n(k)$ such that $f(x)=PyP^{-1}$. Moreover we must have $f(d^{-1}A^{-1}x^tAx)=I_n$ and then $y^{-1}=d^{-1}(P^tAP)^{-1}y^t(P^tAP)$. But we already have $y^{-1}=S(y)=d^{-1}C^{-1}y^tC$. Since the elements $x_{ij}$ are linearly independent, there exists $\lambda \in k^*$ such that $C=\lambda P^tAP$.

Similar computations on the relation $xBx^t=Bd$, using the relations $xd(DC)=(DC)dx$ and $x^td(CD)=(CD)dx^t$, lead to $D=\mu P^{-1t}BP^{-1}, \mu \in k^*$.
Since $\mathcal{G}(A,B)=\mathcal{G}(\alpha A, \beta B)$ for all $\alpha, \beta \in k^*$, we can drop $\lambda$ and $\mu$.


 If $f_*(U_{(0,1)}^{AB})\simeq U_{(0,-1)}^{CD}$ (i.e. if $f(d)=d^{-1}$), then $f_*(U_{(1,0)}^{AB})\simeq U_{(1,-1)}^{CD}$.

  In this case, $m=n$ and there exists $M\in GL_n(k)$ such that $f(x)=Myd^{-1}M^{-1}$. 
Similar computations lead to $C=\lambda P^tB^{-1}P$ and $D=\mu P^{-1}A^{-1}P^{-1t}$, for some $\lambda, \mu \in k^*$.

\end{proof}
We are now ready to prove Theorem \ref{T 2}. 

\begin{proof}[Proof of Theorem \ref{T 2}]

First, for matrices $A,B\in GL_n(k)$ ($n\geq 2$) satisfying the conditions of the theorem, Theorem \ref{T 1} ensures that the Hopf algebra $\mathcal{G}(A,B)$ is indeed a $GL(2)$-deformation.

Let $H$ be a Hopf algebra whose corepresentation semi-ring is isomorphic to that of $GL_2(k)$. We denote by $U_{(n,e)}^H,  (n,e)\in \Bbb N \times \Bbb Z$ the simple $H$-comodules (with the same convention as above). From the morphisms $$U_{(1,0)}^H\otimes U_{(1,0)}^H \to U_{(0,1)}^H \text{  and  } U_{(0,1)}^H\to U_{(1,0)}^H\otimes U_{(1,0)}^H$$  we deduce the existence of two matrices $A,B\in GL_n(k)$ ($n=\dim U_{(1,0)}^H$) and of a Hopf algebra morphism $$f: \mathcal{G}(A,B)\to H$$ such that $f_*(U_{(0,1)}^{AB})=U_{(0,1)}^H$ and $f_*(U_{(1,0)}^{AB})=U_{(1,0)}^H$ and by Lemma \ref{L Schur} there exists $\lambda \in k^*$ such that $B^tA^tBA=\lambda I_n$ for some $\lambda \in k^*$. By Theorem \ref{T 1},  there is a $k$-linear equivalence of monoidal categories $$\textrm{Comod}(\mathcal{G}(A,B))\simeq^{\otimes}\textrm{Comod}(\mathcal{O}(GL_q(2))$$ between the comodule categories of $\mathcal{G}(A,B)$ and $\mathcal{O}(GL_q(2))$ respectively, with $q\in k^*$ such that $tr(AB^t)=\sqrt{\lambda}(q+q^{-1})$.

First assume that $\mathcal{G}(A,B)$ is cosemisimple. Using lemma \ref{L 2}, we get that $f_*(U_{(n,e)}^{AB})=U_{(n,e)}^{H}, \forall \ (n,e)\in \Bbb N \times \Bbb Z$, so $f$ induces a semi-rings isomorphism $\mathcal{R}^+(\mathcal{G}(A,B))\simeq \mathcal{R}^+(H)$, and then by Tannaka-Krein reconstruction theorem (see e.g. \cite{JS}) $f: \mathcal{G}(A,B)\to H$ is a Hopf algebra isomorphism.

Now assume that $\mathcal{G}(A,B)$ is not cosemisimple. For $(n,e)\in \{0,\dots, N_0-1\} \times \Bbb Z$, we have $f_*(U_{(n,e)}^{AB})=U_{(n,e)}^{H}$. So we get:
$$f_*(U_{(N_0-1,0)}^{AB}\otimes U_{(1,0)}^{AB})\simeq U_{(N_0,0)}^H \oplus U^H_{(N_0-2,1)}$$
but on the other hand, using the simple filtration, we have:
$$f_*(U_{(N_0-1,0)}^{AB}\otimes U_{(1,0)}^{AB})\simeq U^H_{(N_0-2,1)}\oplus f_*(V_1) \oplus U^H_{(N_0-2,1)}$$ This contradicts the uniqueness of the decomposition of a semisimple comodule into a direct sum of simple comodules.

Thus $\mathcal{G}(A,B)$ is cosemisimple, $q$ is generic and $f$ is an isomorphism.
\end{proof}

Lemma \ref{L 2} and the results of Section 3 gives us a Morita equivalence criterion which, in the particular case of $\mathcal{O}(GL_{p,q}(2))$, gives Theorem 2.6 in \cite{Tak97}, at the Hopf algebra level.

\begin{cor}\label{Cor1}
Let $A,B \in GL_n(k)$, $C,D\in GL_m(k)$ such that $B^tA^tBA=\lambda_{A,B}I_n$ and $D^tC^tDC=\lambda_{C,D}I_m$. Put $\mu_{A,B}:=\emph{tr}(AB^t)$ and $\mu_{C,D}:=\emph{tr}(CD^t)$. The following assertions are equivalent:
\begin{enumerate}
 \item There exists a $k$-linear equivalence of monoidal categories $$\emph{\textrm{Comod}}(\mathcal{G}(A,B))\simeq^{\otimes}\emph{\textrm{Comod}}(\mathcal{G}(C,D))$$ between the comodule categories of $\mathcal{G}(A,B)$ and $\mathcal{G}(C,D)$ respectively.
 \item We have $$\lambda_{A,B}^{-1} \mu_{A,B}^2 = \lambda_{C,D}^{-1} \mu_{C,D}^2 $$ 
\end{enumerate}

\end{cor}

\begin{proof}
First, put $\kappa:=\lambda_{A,B}^{-1} \mu_{A,B}^2 = \lambda_{C,D}^{-1} \mu_{C,D}^2$ and let $q\in k^*$ such that $q^2-\sqrt{\kappa}q+1=0$. Then by Theorem \ref{T 1} and its proof, we have two $k$-linear equivalences of  monoidal categories
$$\textrm{Comod}(\mathcal{G}(A,B))\simeq^{\otimes}\textrm{Comod}(\mathcal{O}(GL_q(2)))\simeq^{\otimes}\textrm{Comod}(\mathcal{G}(C,D))$$

 For the other implication, first assume that the $k$-linear monoidal functor $\Omega : \textrm{Comod}(\mathcal{G}(A,B)) \to \textrm{Comod}(\mathcal{G}(C,D))$ satisfies $\Omega (D_{AB}) \simeq D_{CD}$. Let $(v_i^{AB})_{1 \leq i \leq n}$, $d_{AB}$ and $(v_i^{CD})_{1 \leq i \leq m}$, $d_{CD}$ be some bases of $V_{AB}$, $D_{AB}$ and $V_{CD}$, $D_{CD}$ respectively such that the fundamental colinear maps 
$$\begin{aligned}
 &a:V_{AB}\otimes V_{AB} \to D_{AB} \ \ \ \ \ c:V_{CD}\otimes V_{CD} \to D_{CD} \\
 &b: D_{AB} \to V_{AB}\otimes V_{AB} \ \ \ \ \ d: D_{CD} \to V_{CD}\otimes V_{CD}
\end{aligned}$$
satisfy $$a(v^{AB}_i\otimes v^{AB}_j)=A_{ij}d_{AB}, \ b(d_{AB})=\sum_{i,j=1}^nB_{ij}v^{AB}_i\otimes v^{AB}_j$$ and $$c(v^{CD}_i\otimes v^{CD}_j)=C_{ij}d_{CD}, \ d(d_{CD})=\sum_{i,j=1}^nD_{ij}v^{CD}_i\otimes v^{CD}_j$$

Since $\Omega$ is monoidal, let $c'$ and $d'$ be the colinear map given by the compositions 
$$\xymatrix{
c':\Omega(V_{AB})\otimes \Omega(V_{AB}) \ar[r]^{\ \ \ \ \sim} \ar@{=}[d]& \Omega(V_{AB}\otimes V_{AB}) \ar[r]^{\ \ \Omega(a)} \ar@{=}[d]& \Omega(D_{AB}) \ar@{=}[d]\\
V_{CD} \otimes V_{CD} \ar[r]& V_{CD} \otimes V_{CD}  \ar[r]& D_{CD}}$$

and 
$$\xymatrix{
d':\Omega(D_{AB}) \ar[r]^{\Omega(b) \  } \ar@{=}[d]& \Omega(V_{AB}\otimes V_{AB}) \ar[r]^{\sim \ \ \ } \ar@{=}[d]& \Omega(V_{AB})\otimes \Omega(V_{AB}) \ar@{=}[d]\\
D_{CD} \ar[r]& V_{CD} \otimes V_{CD}  \ar[r]& V_{CD} \otimes V_{CD}}$$

Then there exists $\alpha, \beta \in k^*$ such that $c'=\alpha c$ and $d'=\beta d$. Since $\Omega$ is $k$-linear, we can compute the colinear map given by the compositions $$D_{CD} \to V_{CD}\otimes V_{CD} \to D_{CD}$$ and $$V_{CD}\otimes D_{CD} \to V_{CD}^{\otimes 3} \to D_{CD}\otimes V_{CD} \to V_{CD}^{\otimes 3} \to V_{CD}\otimes D_{CD} $$ We obtain that $$ \mu_{A,B}=\alpha \beta \mu_{C,D} \text{ and } \lambda_{A,B}= \alpha^2\beta^2 \lambda_{C,D}$$ and then we have the expected equality $$\lambda_{A,B}^{-1} \mu_{A,B}^2 = \lambda_{C,D}^{-1} \mu_{C,D}^2 $$

Finally, if the $k$-linear monoidal functor $\Omega : \textrm{Comod}(\mathcal{G}(A,B)) \to \textrm{Comod}(\mathcal{G}(C,D))$ satisfy $\Omega (D_{AB}) \simeq D^{-1}_{CD}$, compose it with the functor induced by the isomorphism $\mathcal{G}(C,D)\simeq \mathcal{G}(D^{-1},C^{-1})$. We get an equivalence of monoidal categories $$\tilde \Omega : \textrm{Comod}(\mathcal{G}(A,B)) \to \textrm{Comod}(\mathcal{G}(D^{-1},C^{-1}))$$ satisfy $\Omega (D_{AB}) \simeq D_{D^{-1},C^{-1}}$, and then $$\lambda_{A,B}^{-1} \mu_{A,B}^2 = \lambda_{C,D}^{-1} \mu_{C,D}^2=\lambda_{D^{-1},C^{-1}}^{-1} \mu_{D^{-1},C^{-1}}^2$$
\end{proof}

In particular, we get another proof of Theorem 2.6 in \cite{Tak97}
Recall that $\mathcal{O}(GL_{p,q}(2))=\mathcal{G}(A_p, A_q)$ with $\lambda_{A_p,A_q}=pq$ and $\mu_{A_p,A_q}=1+pq$.

\begin{cor}
 Let $p,q$ and $p',q'\in k^*$. The following assertions are equivalent:
 \begin{enumerate}
 \item The Hopf algebras $\mathcal{O}(GL_{p,q}(2))$ and $\mathcal{O}(GL_{p',q'}(2))$ are cocycle deformations of each others
 \item We have $$pq=p'q' \text{  or  } pq=(p'q')^{-1}$$
\end{enumerate}
\end{cor}

\begin{proof} 
 Assume that $\mathcal{O}(GL_{p,q}(2))$ is a cocycle deformation of $\mathcal{O}(GL_{p',q'}(2))$, then $$Comod (\mathcal{O}(GL_{p,q}(2)))\simeq^{\otimes} Comod(\mathcal{O}(GL_{p',q'}(2)))$$ and according to Corollary \ref{Cor1}, $(pq)^{-1}(1+pq)^2=(p'q')^{-1}(1+p'q')^2$. Then $pq$ and $p'q'$ are roots of the polynomial $P(x)=X^2-\Theta X+1$ where $\Theta=(pq)^{-1}+pq=(p'q')^{-1}+p'q'$. It is easy to see that if $x$ is a root of $P$, then the other root is $x^{-1}$. Then $pq=p'q' \text{  or  } pq=(p'q')^{-1}$.

 Conversely, if $pq=p'q'\text{ or } pq=(p'q')^{-1}$ then $\textrm{Comod}(\mathcal{O}(GL_{p,q}(2)))\simeq^{\otimes} \textrm{Comod}(\mathcal{O}(GL_{p',q'}(2)))$ and the fibre functor $\Omega : \textrm{Comod} (\mathcal{G}(A_p, A_q)) \to \textrm{Comod} (\mathcal{G}(A_{p'}, A_{q'})$ induced by $\mathcal{G}(A_p, A_q| A_{p'}, A_{q'})$ preserves the dimensions of the underlying vector space. Then, according to Proposition 4.2.2 in  \cite{EtGel01}, $$\mathcal{G}(A_{p'}, A_{q'})\simeq \mathcal{G}(A_{p}, A_{q}) \text{  as coalgebras}$$ and there exists (see Theorem 7.2.2 in \cite{mon}) a $2$-cocycle $\sigma : H\otimes H \to k$ such that $$\mathcal{G}(A_{p'}, A_{q'})\simeq \mathcal{G}(A_{p}, A_{q})_{\sigma} \text{  as Hopf algebras.}$$

\end{proof}

\section{Hopf-Galois objects over $\mathcal{G}(A,B)$}


In this section, we use the previous constructions and results to classify the Galois and bi-Galois objects over $\mathcal{G}(A,B)$.

Let us first recall two results on Galois objects and fibre functors. The first one is well-known, and the second one is due to Schneider (see \cite{Sch90}, \cite{Bi2}).

\begin{lemma}\label{L3}
Let $H$ be a Hopf algebra and let $F:\emph{\textrm{Comod}}(H) \to\emph{ \textrm{Vect(k)}}$ be a monoidal functor. If $V$ is a finite-dimensional $H$-comodule, then $F(V)$ is a finite dimensional vector space. Moreover we have $\emph{\textrm{dim}}(V)=1\Rightarrow \emph{\textrm{dim}}(F(V))=1$, and if $F$ is a fibre functor then $\emph{\textrm{dim}}(F(V))=1\Rightarrow\emph{\textrm{dim}}(V)=1$.
\end{lemma}

\begin{lemma}\label{L4}
Let $H$ be a Hopf algebra and let $A,B$ some $H$-Galois objects. Any $H$-colinear algebra map $f:A\to B$ is an isomorphism.
\end{lemma}

By work of Ulbrich \cite{Ul}, to any $H$-Galois objects $A$ corresponds a fibre functor $\Omega_A:\textrm{Comod}_f(H)\to \textrm{Vect}_f(k)$. The idea of the classification (which follows \cite{Aub}) is to study how this fibre functor will transform the fundamental morphisms of the category of comodules.

\begin{theo}\label{T4}
Let $A,B\in GL_n(k)$ ($n\geq 2$), such that $B^tA^tBA=\lambda I_n$ for $\lambda \in k^*$, and let $Z$ be a left $\mathcal{G}(A,B)$-Galois object. Then there exists $m\in \Bbb N^*$, $m\geq 2$, and two matrices $C,D\in GL_m(k)$ satisfying $D^tC^tDC=\lambda I_m$ and $tr(AB^t)=tr(CD^t)$ such that $Z\simeq \mathcal{G}(A,B|C,D)$ as Galois objects.
\end{theo}

\begin{proof}
Let $$\begin{aligned}
\Omega_Z : \textrm{Comod}_f(\mathcal{G}(A,B))&\to \textrm{Vect}_f(k) \\
     V &\mapsto V\square_{\mathcal{G}(A,B)} Z
     \end{aligned}$$
be the monoidal functor associated to $Z$.
Let $V_{AB}$ and $D_{AB}^{\pm 1}$ denote the fundamental comodules of $\mathcal{G}(A,B)$, and let $(v_{i})_{1\leq i \leq n}$ and $d_{AB}^{\pm1}$ be their bases such that the fundamental colinear maps
$$\begin{aligned}
 &a:V_{AB}\otimes V_{AB} \to D_{AB}\\
 &b: D_{AB} \to V_{AB}\otimes V_{AB}
\end{aligned}$$
satisfy $a(v_i\otimes v_j)=A_{ij}d_{AB}$ and $b(d_{AB})=\sum_{i,j=1}^nB_{ij}v_i\otimes v_j$.

Let $(w_i)_{1\leq i\leq m}$ and $d^{\pm 1}$ be respective basis of $\Omega_Z(V_{AB})$ and $\Omega_Z(D_{AB}^{\pm 1})$. By construction of $\Omega_Z$, there exists $(z_{ij})_{1\leq i\leq n, 1\leq j \leq m}$ and $d_Z^{\pm 1}$ (see Lemma \ref{L3}) such that $$w_i=\sum_{k=1}^n v_k\otimes z_{ki} \ \ \ \ d^{\pm 1}=d_{AB}^{\pm 1}\otimes d_{Z}^{\pm 1}$$
Moreover, by definition of the cotensor product we have 
$$\begin{aligned}
\alpha(z_{ij})=\sum_k a_{ik}\otimes z_{kj} \\
\alpha(d_Z^{\pm 1}) = d'^{\pm 1}_{AB} \otimes d_Z^{\pm 1}                                                                         
\end{aligned}
$$
where $a_{ij}$ and $d'^{\pm 1}_{AB}$ denotes the generators of $\mathcal{G}(A,B)$.

Consider the bilinear map defined by the composition 
$$\xymatrix{
a':\Omega_Z(V_{AB})\otimes \Omega_Z(V_{AB}) \ar[r]^{\ \ \ \ \sim} \ar@{=}[d]& \Omega_Z(V_{AB}\otimes V_{AB}) \ar[r]^{\ \ \ \Omega_Z(a)} \ar@{=}[d]& \Omega_Z(D_{AB}) \ar@{=}[d]\\
( V_{AB}\square_{\mathcal{G}(A,B)} Z ) \otimes (V_{AB}\square_{\mathcal{G}(A,B)} Z ) \ar[r]^{\ \ \ \ \ \ \ \sim}& (V_{AB} \otimes V_{AB})\square_{\mathcal{G}(A,B)} Z  \ar[r]^{\ \ \ \ a\otimes id}& D_{AB}\square_{\mathcal{G}(A,B)} Z}$$

and let $C=(C_{ij})_{1\leq i,j\leq m}$ such that $a'(w_i\otimes w_j)=C_{ij}d$. Then we compute:

$$\begin{aligned}
  a'(w_i\otimes w_j)&=a'((\sum_{k=1}^n v_k\otimes z_{ki})\otimes (\sum_{l=1}^n v_l\otimes z_{lj}))\\
  &=\sum_{k,l} A_{kl}d_{AB}\otimes z_{ki}z_{lj}=C_{ij}d_{AB}\otimes d_{Z}
 \end{aligned}$$

or in matrix form $$z^tAz=Cd_Z$$

In the same way, consider the map 

$$\xymatrix{
b':\Omega_Z(D_{AB}) \ar[r]^{\Omega_Z(b)\ \ } \ar@{=}[d]& \Omega_Z(V_{AB}\otimes V_{AB}) \ar[r]^{\psi \ \ } \ar@{=}[d]& \Omega_Z(V_{AB})\otimes \Omega_Z(V_{AB}) \ar@{=}[d]\\ 
D_{AB}\square_{\mathcal{G}(A,B)} Z \ar[r]^{b\otimes id\ \ \ \ \ \ } & (V_{AB} \otimes V_{AB})\square_{\mathcal{G}(A,B)} Z \ar[r]^{\psi \ \ \ \ \ } & (V_{AB}\square_{\mathcal{G}(A,B)} Z) \otimes (V_{AB}\square_{\mathcal{G}(A,B)} Z)}$$

Let $D=(D_{ij})_{1\leq i,j\leq m}$ be defined by $b'(d)=\sum D_{ij} w_i\otimes w_j$.

Then we have:
$$\begin{aligned}
\psi^{-1}\circ b'(d_{AB}\otimes d_Z)&=b\otimes id (d_{AB}\otimes d_Z)\\
&=\sum_{i,j} B_{ij} v_i\otimes v_j \otimes d_Z
\end{aligned}$$
and
$$\begin{aligned}
 \psi^{-1}\circ b'(d)&=\psi^{-1}(\sum_{ij}D_{ij}(\sum_k v_k\otimes z_ki) \otimes (\sum_l v_l\otimes z_{lj})\\
 &=\sum_{kl} \sum_{ij} v_k\otimes v_l \otimes D_{ij}z_{ki}z_{lj}
\end{aligned}$$

so $$zDz^t=Bd_Z$$
Hence we have an algebra morphism $f:\mathcal{G}(A,B|C,D) \to Z$ defined by $f(x)=z$ and $f(d^{\pm 1})=d_Z^{\pm 1}$

We have to check that $f$ is colinear. Since it is an algebra map, it is sufficient to check on the generators, which is trivial by the construction of respective coactions and by the definition of $f$. Then by Lemma \ref{L4}, $f$ is an isomorphism.

Finally, Schur's lemma gives the equality $$ (a \otimes id) \circ (id\otimes b) \circ (id\otimes a)\circ(b \otimes id)=\lambda id_{D_{AB}\otimes V_{AB}} \ (\lambda \in k^*)$$ which may be rewritten in matrix form as $$B^tA^tBA=\lambda I_n$$ Since the functor $\Omega_Z$ is $k$-linear, we have $$ (a' \otimes id) \circ (id\otimes b') \circ(id\otimes a')\circ(b' \otimes id)=\lambda id_{\Omega_Z(D_{AB})\otimes \Omega_Z(V_{AB})}$$ and then $$D^tC^tDC=\lambda I_m$$ and we have $a\circ b = \text{tr}(AB^t)id_{D_{AB}}$, so, similarly, we have $\text{tr}(AB^t)=\text{tr}(CD^t)$.
\end{proof}

\begin{theo}\label{T 5}
Let $A,B\in GL_n(k)$ such that $B^tA^tBA=\lambda I_n$ and let $C_1,D_1 \in GL_{m_1}(k)$, $C_2,D_2\in GL_{m_2}(k)$ such that the algebras $\mathcal{G}(A,B|C_1,D_1)$ and $\mathcal{G}(A,B|C_2,D_2)$ are $\mathcal{G}(A,B)$-Galois objects ($n,m_1$, and $m_2 \geq 2$). Then $\mathcal{G}(A,B|C_1,D_1)$ and $\mathcal{G}(A,B|C_2,D_2)$ are isomorphic (as Galois object) if and only if $m_1=m_2:=m$ and there exists an invertible matrix $M\in GL_{m}(k)$ such that $(C_2,D_2)=(M^{-1t}C_1M^{-1},MD_1M^t)$.
\end{theo}

\begin{proof}
We denote by $\Omega_i$ the fibre functor associated to $\mathcal{G}(A,B|C_i,D_i)$.
Let $f: \mathcal{G}(A,B|C_1,D_1)\to \mathcal{G}(A,B|C_2,D_2)$ be a comodule algebra isomorphism: it induces an isomorphism $id\otimes f : \Omega_1(U_{AB})\to \Omega_2(U_{AB})$. Using the same notation as above, we get two basis $(w_i^1)_{1\leq i\leq m_1}$ and $(w_i^2)_{1\leq i\leq m_2}$ of $\Omega_1(U_{AB})$ and $\Omega_2(U_{AB})$. In particular, we have $m_1=m_2:=m$. Then there exists $M=(M_{ij})\in GL_m(k)$ such that $id\otimes f(w_i^1)=\sum_k M_{ji}w_j^2$, and hence 
$$\sum_k v_k\otimes f(z^1_{ki})=\sum_k v_k \otimes z_{kj}^2M_{ji}$$ which in matrix form gives $f(z^1)=z^2M$.

According to the relations defining $\mathcal{G}(A,B|C_1,D_1)$ we have 
$$(z^1)^tAz^1=C_1d \text{   and   } z^1D_1(z^1)^t=Bd$$
hence $$f((z^1)^tAz^1)=M^t(z^2)^tAz^2M=M^tC_2Md=f(C_1d)=C_1d \text{  in  } \mathcal{G}(A,B|C_1,D_1)$$ 
so $M^tC_2M=C_1$, and the second relation leads to 
$$f(z^1D_1(z^1)^t)=z^2MD_1M^t(z^2)^t=f(Bd)=Bd=z^2D_2(z^2)^t$$
so $D_2=MD_1M^t$.

Conversely, we already have $\mathcal{G}(A,B|C,D)\simeq \mathcal{G}(A,B|M^{-1t}CM^{-1},MDM^t)$, see Proposition \ref{P 2}.
\end{proof}

According to the work of Schauenburg \cite{Sch1}, the set of bi-Galois objects $\textrm{BiGal}(L,H)$ is a groupoid with multiplication given by the cotensor product. In particular, when $H=L$, the set of isomorphism classes of $H$-$H$-bi-Galois objects inherits a structure of groups. Then, we have two group morphisms $$\textrm{Aut}_{\textrm{Hopf}}(H)\to \textrm{BiGal}(H), \ f \mapsto [H^f]$$ with kernel $ \textrm{CoInn}(H):=\{f\in \textrm{Aut}_{\textrm{Hopf}}(H); \exists \phi \in Alg(H,k) \text{ with } f=(\phi \circ S)\star id_H \star \phi \}$ and we denote $ \textrm{CoOut}(H):= \textrm{Aut}_{\textrm{Hopf}}(H)/ \textrm{CoInn}(H)$; and $$H^2_{\ell}(H) \to \textrm{BiGal}(H), \ \sigma \mapsto [H(\sigma)]$$ where $H^2_{\ell}(H)$ denotes the lazy cohomology group of $H$, see \cite{BC}. From the monoidal categories viewpoint, it is the subgroup of $ \textrm{BiGal}(H)$ consisting of isomorphism classes of linear monoidal auto-equivalences of the category of $A$-comodules that are isomorphic, as functors, to the identity functor.

We assume until the end of the section that $k$ has characteristic zero.

\begin{lemma}
The automorphism group $\emph{\textrm{Aut}}_{\emph{\textrm{Hopf}}}(\mathcal{G}(A,B))$ is isomorphic with the group $$\emph{\textrm{G}}_{(A,B)}=\{P\in GL_n(k); A=P^tAP, B=P^{-1}BP^{-1t} \text{ or } A=P^tB^{-1}P, B=P^{-1}A^{-1}P^{-1t})\}/\{\pm I_n\}$$

Moreover, we have $$\emph{\textrm{CoInn}} (\mathcal{G}(A,B)) \simeq \{P\in GL_n(k); A=P^tAP, B=P^{-1}BP^{-1t}\}/\{\pm I_n\}$$ and $$\emph{\textrm{CoOut}} (\mathcal{G}(A,B)) \simeq \Bbb Z / 2\Bbb Z$$
\end{lemma}

\begin{proof}
The first isomorphism comes from the proof of Theorem \ref{T 3}, and the assertion about CoInn is easy to verify. Finally, $\textrm{CoOut} (\mathcal{G}(A,B)) \simeq \Bbb Z / 2\Bbb Z$ because for any $f,g \in \textrm{Aut}_{\textrm{Hopf}}(\mathcal{G}(A,B)) \setminus \textrm{CoInn}(\mathcal{G}(A,B))$, $f\circ g \in \textrm{CoInn}(\mathcal{G}(A,B))$.
\end{proof}

\begin{theo}
For any $n\geq 2$ and $A,B \in GL_n(k)$ such that $B^tA^tBA=\lambda I_n$ ($\lambda \in k^*$), $$\emph{\textrm{BiGal}}(\mathcal{G}(A,B))\simeq \Bbb Z / 2\Bbb Z$$
\end{theo}

\begin{proof}
Let $Z$ be a $\mathcal{G}(A,B)$-$\mathcal{G}(A,B)$-bi-Galois object. By Theorem \ref{T4}, there exists $m\geq 2$ and $C,D\in GL_m(k)$ verifying $D^tC^tDC=\lambda I_m$ and $tr(AB^t)=tr(CD^t)$ such that $$Z\simeq \mathcal{G}(A,B|C,D)$$ as a $\mathcal{G}(A,B)$-Galois object. Since $\mathcal{G}(A,B|C,D)$ is also a $\mathcal{G}(A,B)$-$\mathcal{G}(C,D)$-bi-Galois object, the Hopf algebras $\mathcal{G}(A,B)$ and $\mathcal{G}(C,D)$ are isomorphic (by \cite{Sch1}, Theorem 3.5), so, by Theorem \ref{T 3}, $m=n$ and there exists $P\in GL_n(k)$ such that $(C,D)\in \{(P^tAP,  P^{-1}BP^{-1t}),( P^tB^{-1}P, P^{-1}A^{-1}P^{-1t})\}$. Then we have either $$Z\simeq \mathcal{G}(A,B|C,D) \simeq \mathcal{G}(A,B)$$ or $$Z\simeq \mathcal{G}(A,B|C,D) \simeq \mathcal{G}(A,B|B^{-1}, A^{-1})$$ as left Galois objects. Moreover, according to \cite{Sch1}, Lemma 3.11, $\textrm{CoOut} (\mathcal{G}(A,B))$ acts freely on $\textrm{BiGal}(\mathcal{G}(A,B))$ by $$f \in \textrm{CoOut} (\mathcal{G}(A,B)), \ A \in \textrm{BiGal}(\mathcal{G}(A,B)) \ : f.A = A^f$$
Then we have to check that $$\mathcal{G}(A,B|B^{-1}, A^{-1}) \simeq \mathcal{G}(A,B)^f$$ where $f \in \textrm{CoOut} (\mathcal{G}(A,B))$ is non trivial. To do so, it is easy to verify that $$\Omega_{\mathcal{G}(A,B|B^{-1}, A^{-1})}(D_{AB})\simeq D^{-1}_{AB} \simeq \Omega_{\mathcal{G}(A,B)^f}(D_{AB})$$ where $\Omega_Z$ denote the fiber functor induced by $Z$. Then by Lemma \ref{L 2}, the functors are isomorphic, and according to Ulbrich's work \cite{Ul}, the bi-Galois objects are isomorphic.
\end{proof}

Finally, from the interpretation of bi-Galois objects as functor, we get:
\begin{theo}
For any $n\geq 2$ and $A,B \in GL_n(k)$ such that $B^tA^tBA=\lambda I_n$ ($\lambda \in k^*$), $H_{\ell}^2(\mathcal{G}(A,B))$ is trivial.
\end{theo}

In particular, according to \cite{BC}, Theorem 3.8, $\mathcal{G}(A,B)$ has no non-trivial bi-cleft bi-Galois object.

\section{Hopf $*$-algebras structure on $\mathcal{G}(A,B)$}

In this section, $k=\Bbb C$. We classify CQG algebras which are $GL(2)$-deformations (or rather $U(2)$-deformations). 

Let us recall that a Hopf $*$-algebra is a Hopf algebra $H$ which is also a $*$-algebra and such that the comultiplication is a $*$-homomorphism. If $x=(x_{ij})_{1\leq i,j\leq n}\in M_n(H)$ is a matrix with coefficient in $H$, the matrix $(x^*_{ij})_{1\leq i,j\leq n}$ is denoted by $\overline{x}$, while $\overline{x}^t$, the transpose matrix of $\overline{x}$, is denoted by $x^*$. The matrix $x$ is said to be unitary if $x^*x=I_n=xx^*$. Recall (\cite{KS}) that a Hopf $*$-algebra is said to be a CQG algebra if for every finite-dimensional $H$-comodule with associate matrix $x\in M_n(H)$, there exists $K\in GL_n(\Bbb C)$ such that the matrix $KxK^{-1}$ is unitary. CQG algebras correspond to Hopf algebras of representative functions on compact quantum groups.

We begin with a lemma which gives an example of CQG algebra structure on $\mathcal{G}(A,B)$.

\begin{lemma}
Let $E\in GL_n(\Bbb C)$ such that $\overline{E}^tE^t\overline{E}E=\lambda I_n$ for $\lambda \in \Bbb C$. Then $\lambda \in \Bbb R^*_+$ and the Hopf algebra $\mathcal{G}(E,\overline{E})$ is a CQG algebra for the following $*$-algebra structure:
$$d^*=d^{-1}   \text{  and  } \overline{x}=E^td^{-1}xE^{-1t}$$
The CQG algebra $\mathcal{G}(E,\overline{E})$ will be denoted by $A_{\tilde{o}}(E)$.
\end{lemma}

\begin{proof}
First, notice that because of the relations defining $\mathcal{G}(E,\overline{E})$ and the condition on $E$, we also have $\overline{x}=\overline{E}^{-1t}xd^{-1}\overline{E}^{t}$. Then we can verify that our structure is well defined: 
for the first relation, we compute: 
$$\begin{aligned}
  \overline{E^{-1}x^tEx}&=((\overline{Ex})^t(\overline{E^{-1}x^t})^t)^t \\
  &=((\overline{E}E^td^{-1}xE^{-1t})^t(\overline{E^{-1}}(\overline{E}^{-1t}xd^{-1}\overline{E}^{t})^t)^t)^t\\
  &=((E^{-1}d^{-1}x^tE\overline{E}^t)(\overline{E}^{-1t}xd^{-1}\overline{E}^{t}\overline{E^{-1t}}))^t=d^{-1}
 \end{aligned}$$
and for the second one we get:
$$\begin{aligned}
\overline{\overline{E}^{-1}x\overline{E}x^t}&=((E\overline{x}^t)^t(E^{-1}\overline{x})^t)^t \\
&=((E(E^td^{-1}xE^{-1t})^t)^t(E^{-1}\overline{E}^{-1t}xd^{-1}\overline{E}^{t})^t)^t \\
&=((E^td^{-1}xE^{-1t}E^t)(\overline{E}x^td^{-1}\overline{E}^{-1}E^{-1t}))^t= d^{-1}
 \end{aligned}$$
Let us show that we have a $*$-structure and that $x$ is unitary: 
first $$\overline{\overline{x}}=\overline{E^td^{-1}xE^{-1t}}=\overline{E}^t\overline{x}d\overline{E}^{-1t}=\overline{E}^t\overline{E}^{-1t}xd^{-1}\overline{E}^{t}d\overline{E}^{-1t}=x$$
and finally, we have 
$$\begin{aligned}
x^*=\overline{x}^t&=(\overline{E}^{-1t}xd^{-1}\overline{E}^{t})^t= \overline{E} x^td^{-1}\overline{E}^{-1}\\
&=(E^td^{-1}xE^{-1t})^t=E^{-1}d^{-1}x^tE
\end{aligned}$$
so according to the relations defining $\mathcal{G}(E,\overline{E})$ we have $x^*x=xx^*=I_n$, $d^*d=dd^*=1$ and, by \cite{KS}, $\mathcal{G}(E,\overline{E})$ is CQG.

Finally, we have $\overline{E}^tE^t\overline{E}E=\lambda I_n=(E^tE^{t*})(EE^*)$, so $\lambda \in \Bbb R^*_+$.
\end{proof}

The terminology $A_{\tilde{o}}(E)$ follows from the recent paper \cite{BBCC}, where $ \tilde{O}_n$ denotes the subgroup of $U_n(\Bbb C)$ generated by $O_n(\Bbb R)$ and $\Bbb T.I_n$.

As a special case of the lemma, we get the following result from \cite{HM}:

\begin{cor}
The Hopf algebra $\mathcal{O}(GL_{q,\overline{q}}(2))$ is a CQG algebra, for the $*$-structure given by $$D^*=D^{-1}\text{ and } \left(\begin{array}{cc}
a^* & b^* \\
c^* & d^* 
\end{array}\right)=\left(\begin{array}{cc}
dD^{-1} & -qcD^{-1} \\
-q^{-1}bD^{-1}  & aD^{-1} 
\end{array}\right)$$
In particular, for $q\in \Bbb R^*$, $\mathcal{O}(GL_{q}(2))$ is CQG.
\end{cor}

We can state and prove the main theorem of this section: 

\begin{theo}\label{T 6}
The CQG algebras whose corepresentation semi-ring is isomorphic that of $U_2(\Bbb C)$ are exactly the $$A_{\tilde{o}}(E)$$ where $E\in GL_n(\Bbb C)$ $(n\geq 2)$ satisfies $\overline{E}^tE^t\overline{E}E=\lambda I_n$ for $\lambda \in \Bbb R^*_+$. 
\end{theo}

\begin{proof}
First of all, the algebra $A_{\tilde{o}}(E)$ are indeed $U(2)$-deformations, according to the previous lemma and to Theorem \ref{T 1}.

Let $H$ be a CQG algebra such that $\mathcal{R}^+(H)\simeq \mathcal{R}^+(\mathcal{O}(U(2))$. Let denote by $d_H, d^{-1}_H$ and $x=(x_{ij})_{1\leq i,j \leq n}$ ($2\leq n$) the matrix coefficients of $U_{(0,1)}, U_{(0,-1)}$ and $U_{(1,0)}$ respectively. Since $H$ is a CQG algebra, we have $d_H^*=d_H^{-1}$ and we can assume that the matrix $x$ is unitary. Lemma \ref{L 2} and its proof gives us $\overline{U_{(1,0)}^H} \simeq U^H_{(1,-1)} \simeq U^H_{(0,-1)} \otimes U^H_{(1,0)}$, hence there exist $F,G\in GL_n(\Bbb C)$ ($n=\dim_{\Bbb C} U_{(1,0)}^H$) such that 
$$x=F\overline{x}dF^{-1},  \ \ x=Gd\overline{x} G^{-1} \text{   and   }  xx^*=I_n=x^*x$$ where $\overline{x}=(x^*_{ij})_{1\leq i,j\leq n}$ and $x^*=\overline{x}^t$.
We have $$x=\overline{\overline{x}}=\overline{G^{-1}}F^{-1}xF\overline{G}$$ hence we get $$\overline{F}G=\nu I_n\text{  for some  } \nu \in \Bbb C^*$$ and using the relations $xx^*=I_n=x^*x$ we get : $$xF^tx^t=dF^t \text{  and  } x^tG^{-1t}x=dG^{-1t}$$ We put $E=\overline{F}^{t}$ and using the universal propertie of $A_{\tilde{o}}(E)=\mathcal{G}(E,\overline{E})$, we get a Hopf $*$-algebra morphism $$ f : A_{\tilde{o}}(E) \to H$$
such that 
$$f(d)=d_H, \ \ f(d^{-1})=d^{-1}_H , \ \ f(x)=x_H$$ 
Since $H$ is cosemisimple, the matrices $F$ and $G$ must satify $G^{-1t}F^tG^{-1}F=\mu I_n$ with $\mu \in \Bbb C^*$. Then $E$ satisfies $\overline{E}^tE^t\overline{E}E=\lambda I_n=(E^tE^{t*})(EE^*)$ for $\lambda \in \Bbb R^*_+$. So we know from Theorem \ref{T 1} that the corepresentation semi-ring of $A_{\tilde{o}}(E)$ is isomorphic to that of $U(2)$, hence $f$ induces an isomorphism of semi-ring between $\mathcal{R}^+(A_{\tilde{o}}(E))$ and $\mathcal{R}^+(H)$. We conclude by Tannaka-Krein reconstruction techniques that $f : A_{\tilde{o}}(E) \to H$ is a Hopf $*$-algebra isomorphism. 

\end{proof}

\appendix

\section*{Appendix: proof of Lemma \ref{L 1}}

This section is devoted to the proof of Lemma \ref{L 1}. The strategy of our proof is to write a convenient presentation of the algebra $\mathcal{G}(A_q,A_q|C,D)$ so that we can apply the diamond lemma (Bergman, \cite{Ber}) to get some linearly independent elements: this will imply that $\mathcal{G}(A_q,A_q|C,D)$ is non zero.

Recall that the algebras $\mathcal{G}(A,B|C,D)$ and $\mathcal{G}(P^tAP,P^{-1}BP^{-1t}|Q^tCQ,Q^{-1}DQ^{-1t})$
are isomorphic by Proposition \ref{P 1}. Combining this fact with the following well known lemma, and we can assume that $D_{mm}=0$:
\begin{lemmaa}
Let $M\in GL_n(k)$ ($n\geq 2$). Then there exist a matrix $P\in GL_n(k)$ such that $(P^tMP)_{nn}=0$.
\end{lemmaa}

Let us now study in detail the algebra $\mathcal{M}(A_q,A_q|C,D)$: it is the universal algebra with generators $x_{ij}$, $1\leq i \leq 2, 1\leq j\leq m$ and $d$, with relations $$ x^tA_qx=Cd \textbf{ (1) }, \  \ \ xDx^t=A_qd \textbf{ (2)}$$ 

We can write these relations explicitly: 

$\begin{aligned}
& x_{2i}x_{1j}=q^{-1}(x_{1i}x_{2j}-C_{ij}d) ,\ 1 \leq i \leq 2, 1\leq j\leq m, \ \textbf{(1')}\\
& \sum_{k,l=1}^mD_{kl}x_{1k}x_{2l}=d, &\textbf{(2')}\\
& \sum_{k,l=1}^mD_{kl}x_{1k}x_{1l}=0, &\textbf{(3')}\\
& \sum_{k,l=1}^mD_{kl}x_{2k}x_{2l}=0, &\textbf{(4')}\\
& \sum_{k,l=1}^mD_{kl}x_{2k}x_{1l}=qd, &\textbf{(5')}\\
\end{aligned}$

Using the fact that $\sum_{k,l=1}^mC_{kl}D_{kl}=1+q^2$, we see that relations $\textbf{(1')}$ and $\textbf{(2')}$ imply relation $\textbf{(5')}$. We will also need to get commutation relations between $d$ and the $x_{ij}$: note that relation $\textbf{(1)}$ and $\textbf{(2)}$ imply  $$\begin{aligned}
  &x^tdA_q^2=CDdx^t\\
  &xdDC=A_q^2dx
 \end{aligned}$$
which gives us 

$\\ \begin{aligned}
x_{1j}d&=-q\sum_{k=1}^m(C^{-1}D^{-1})_{kj}dx_{1k} \ 1\leq j \leq m\\
x_{2j}d&=-q^{-1}\sum_{k=1}^m(CD)_{jk}dx_{2k} \ 1\leq j \leq m
\end{aligned}$

Let us order the set $\{1, 2\} \times \{1,\dots, m\}$ lexicographically. Take $(u,v)$ the maximal element such that $D_{uv}\neq 0$. Since the matrix $D$ is invertible, we have $u=m$ and since $D_{mm}=0$, we have $v<m$. We see now that $\mathcal{M}(A_q,A_q|C,D)$ is the universal algebra with generators $x_{1j}, 1\leq j\leq m$, $x_{2j}, 1\leq j\leq m$ and $d$, with relations 

$\left\{\begin{aligned}
x_{2i}x_{1j}&=q^{-1}(x_{1i}x_{2j}-C_{ij}d) &\textbf{(1)}\\
x_{1m}x_{2v}&=(D_{mv})^{-1}\big(d-\sum_{(kl)<(mv)}D_{kl}x_{1k}x_{2l}\big) &\textbf{(2)}\\
x_{1m}x_{1v}&=-(D_{mv})^{-1}\big(\sum_{(kl)<(mv)}D_{kl}x_{1k}x_{1l}\big) &\textbf{(3)} \\
x_{2m}x_{2v}&=-(D_{mv})^{-1}\big(\sum_{(kl)<(mv)}D_{kl}x_{2k}x_{2l}\big) &\textbf{(4)}\\
x_{1j}d&=-q\sum_{k=1}^m(C^{-1}D^{-1})_{kj}dx_{1k} &\textbf{(5)}\\
x_{2j}d&=-q^{-1}\sum_{k=1}^m(CD)_{jk}dx_{2k} &\textbf{(6)}
\end{aligned} \right.$

We now have a nice presentation to use the diamond lemma (Bergman \cite{Ber}). We use the simplified exposition in the book Klimyk and Schm\"{u}dgen \cite{KS} and freely use the techniques and definitions involved. We endow the set $\{x_{ij}, (i,j) \in\{1,2\}\times\{1,\dots,m\} \}$ with the order induced by the lexicographic order on the set $\{1, 2\} \times \{1,\dots, m\}$, we put $d<x_{ij}$ and we order the set of monomials  according to their length, and finally two monomials of the same length are ordered lexicographically. It is clear that the presentation above is compatible with the order. Hence we have:

\begin{lemmaa}
There are no inclusions ambiguities, and we have exactly the following overlap ambiguities:
$$\begin{aligned}
  &(x_{2i}x_{1m},x_{1m}x_{1v}), &(x_{2i}x_{1m},x_{1m}x_{2v}), &\ \forall 1\leq i \leq m\\
  &(x_{1m}x_{2v},x_{2v}x_{1j}), &(x_{2m}x_{2v},x_{2v}x_{1j}), &\ \forall 1\leq j \leq m\\
  &(x_{2i}x_{1j},x_{1j}d) & &\forall 1\leq i,j \leq m\\
  &(x_{1m}x_{2v},x_{2v}d), &(x_{2m}x_{2v},x_{2v}d) &\\
  &(x_{1m}x_{1v},x_{1v}d) & &
 \end{aligned}$$
These ambiguities are resolvable.
\end{lemmaa}

\begin{proof}
Let us first note some identities:
\[ \begin{aligned}
&(CD)_{ij}=q^2(C^{-1}D^{-1})_{ji}\\
&\sum_{(kl)<(mv)} C_{kl}D_{kl}=1+q^2-C_{mv}D_{mv}\\
&\sum_{(kl)<(mv)} D_{kl}C_{ik}dx_{il}=\sum_{k=1}^m (CD)_{il}dx_{il}-D_{mv}C_{im}dx_{iv}, &\forall 1\leq i \leq m\\
&\sum_{(kl)<(mv)} (C^{-1}D^{-1})_{jk}D_{kl}(CD)_{li}=D_{ji}-(C^{-1}D^{-1})_{jm}D_{mv}(CD)_{vi}, &\forall 1\leq i,j \leq m\\
\end{aligned}\]

Let us show that the ambiguity $(x_{2i}x_{1m},x_{1m}x_{1v})$ is resolvable (the symbol ``$\to$'' means that we perform a reduction).

On the first hand we have:
\begin{flushleft}
$\begin{aligned}
&q^{-1}(x_{1i}x_{2m}x_{1v}-C_{im}dx_{1v})\\
\to &q^{-1}(q^{-1}(x_{1i}x_{1m}x_{2v}-C_{mv}x_{1i}d)-C_{im}dx_{1v})\\
\to &q^{-1}(q^{-1}((D_{mv})^{-1}\big(x_{1i}d-\sum_{(kl)<(mv)}D_{kl}x_{1i}x_{1k}x_{2l}\big)-C_{mv}x_{1i}d)-C_{im}dx_{1v})\\
= &-q^{-1}(D_{mv})^{-1}(q^{-1}(\big(-x_{1i}d+\sum_{(kl)<(mv)}D_{kl}x_{1i}x_{1k}x_{2l}\big)+D_{mv}C_{mu}x_{1i}d)+D_{mv}C_{im}dx_{1v})\\
= &-q^{-1}(D_{mv})^{-1}(q^{-1}\sum_{(kl)<(mv)}D_{kl}x_{1i}x_{1k}x_{2l}-q^{-1}(1-D_{mv}C_{mv})x_{1i}d)+D_{mv}C_{im}dx_{1v})\\
\to &-q^{-1}(D_{mv})^{-1}(q^{-1}\sum_{(kl)<(mv)}D_{kl}x_{1i}x_{1k}x_{2l}-q^{-2}(1-D_{mv}C_{mv})(\sum_{k=1}^m(CD)_{ik}dx_{1k}))+D_{mv}C_{im}dx_{1v})\\
\end{aligned}$
\end{flushleft}

On the other hand: 
\begin{flushleft}
$\begin{aligned}
&-(D_{mv})^{-1}\big(\sum_{(kl)<(mv)}D_{kl}x_{2i}x_{1k}x_{1l}\big)\\
\to &-q^{-1}(D_{mv})^{-1}\big(\sum_{(kl)<(mv)}D_{kl}(x_{1i}x_{2k}-C_{ik}d)x_{1l}\big)\\
= &-q^{-1}(D_{mv})^{-1}\big(\sum_{(kl)<(mv)}D_{kl}x_{1i}x_{2k}x_{1l}-\sum_{k=1}^m (CD)_{il}dx_{1l}+D_{mv}C_{im}dx_{1v}\big)\\
\to &-q^{-1}(D_{mv})^{-1}\big(q^{-1}\sum_{(kl)<(mv)}D_{kl}x_{1i}(x_{1k}x_{2l}-C_{kl}d)-\sum_{k=1}^m (CD)_{il}dx_{1l}+D_{mv}C_{im}dx_{1v}\big)\\
\end{aligned}$
$\begin{aligned}
= &-q^{-1}(D_{mv})^{-1}\big(q^{-1}\sum_{(kl)<(mv)}D_{kl}x_{1i}x_{1k}x_{2l}-q^{-1}\sum_{(kl)<(mv)}D_{kl}C_{kl}x_{1i}d)\\
&-\sum_{l=1}^m (CD)_{il}dx_{1l}+D_{mv}C_{im}dx_{1v}\big)\\
= &-q^{-1}(D_{mv})^{-1}\big(q^{-1}\sum_{(kl)<(mv)}D_{kl}x_{1i}x_{1k}x_{2l}-q^{-1}(1+q^2-C_{mv}D_{mv})x_{1i}d)\\ &-\sum_{l=1}^m (CD)_{il}dx_{1l}+D_{mv}C_{im}dx_{1v}\big)\\
\to &-q^{-1}(D_{mv})^{-1}\big(q^{-1}\sum_{(kl)<(mv)}D_{kl}x_{1i}x_{1k}x_{2l}-q^{-1}(1+q^2-C_{mv}D_{mv})(-q\sum_{k=1}^m(C^{-1}D^{-1})_{ki}dx_{1k})\\&-\sum_{l=1}^m (CD)_{il}dx_{1l}+D_{mv}C_{im}dx_{1v}\big)\\
= &-q^{-1}(D_{mv})^{-1}\big(q^{-1}\sum_{(kl)<(mu)}D_{kl}x_{1i}x_{1k}x_{2l} -q^{-2}(1+q^2-C_{mv}D_{mv})(\sum_{k=1}^m(CD)_{ik}dx_{1k})\\&-\sum_{k=1}^m (CD)_{il}dx_{1l}+D_{mv}C_{im}dx_{1v}\big)\\
= &-q^{-1}(D_{mv})^{-1}\big(q^{-1}\sum_{(kl)<(mv)}D_{kl}x_{1i}x_{1k}x_{2l}-q^{-2}(1-C_{mv}D_{mv})(\sum_{k=1}^m(CD)_{ik}dx_{1k})+D_{mv}C_{im}dx_{1v}\big)\\
  \end{aligned}$
\end{flushleft}

Similar computations show that the ambiguity $(x_{2m}x_{2v},x_{2v}x_{1j})$ is resolvable, using the relations $\textbf{(1)}$, $\textbf{(6)}$ and $\textbf{(2)}$.

Let us show that the ambiguity $(x_{1m}x_{2v},x_{2v}x_{1j})$ is resolvable.

On the first hand we have:
\begin{flushleft}
$\begin{aligned}
 &(D_{mv})^{-1}\big(dx_{1j}-\sum_{(kl)<(mv)}D_{kl}x_{1k}x_{2l}x_{1j}\big)\\
 \to &(D_{mv})^{-1}\big(dx_{1j}-q^{-1}\sum_{(kl)<(mv)}D_{kl}x_{1k}(x_{1l}x_{2j}-C_{lj}d)\big)\\
 = &(D_{mv})^{-1}\big(dx_{1j}-q^{-1}(\sum_{(kl)<(mv)}D_{kl}x_{1k}x_{1l}x_{2j}-\sum_{(kl)<(mv)}D_{kl}x_{1k}C_{lj}d)\big)\\
 = &(D_{mv})^{-1}\big(dx_{1j}-q^{-1}(\sum_{(kl)<(mv)}D_{kl}x_{1k}x_{1l}x_{2j}-\sum_{(kl)<(mv)}D_{kl}C_{lj}x_{1k}d)\big)\\
 = &(D_{mv})^{-1}\big(dx_{1j}-q^{-1}(\sum_{(kl)<(mv)}D_{kl}x_{1k}x_{1l}x_{2j}-\sum_{k=1}^m(DC)_{kj}x_{1k}d-D_{mv}C_{vj}x_{1m}d)\big)\\
 \to &(D_{mv})^{-1}\big(dx_{1j}-q^{-1}(\sum_{(kl)<(mv)}D_{kl}x_{1k}x_{1l}x_{2j}+q\sum_{k,l=1}^m(DC)_{kj}(C^{-1}D^{-1})_{lj}dx_{1l}-D_{mv}C_{vj}x_{1m}d)\big)\\
 = &(D_{mv})^{-1}\big(dx_{1j}-q^{-1}(\sum_{(kl)<(mv)}D_{kl}x_{1k}x_{1l}x_{2j}+qdx_{1j}-D_{mv}C_{vj}x_{1m}d)\big)\\
 = &-q^{-1}(D_{mv})^{-1}\big(\sum_{(kl)<(mv)}D_{kl}x_{1k}x_{1l}x_{2j}+D_{mv}C_{vj}x_{1m}d\big)
  \end{aligned}$
\end{flushleft}

On the other hand we have: 
\begin{flushleft}
$$ \begin{aligned}
   q^{-1}(x_{1m}x_{1v}x_{2j}-C_{vj}x_{1m}d)
   \to &-q^{-1}(D_{mv})^{-1}\big(\sum_{(kl)<(mv)}D_{kl}x_{1k}x_{1l}x_{2j}+D_{mv}C_{vj}x_{1m}d)\big)
  \end{aligned}$$
  \end{flushleft}

Similars computations shows that the ambiguity $(x_{2i}x_{1m},x_{1m}x_{2v})$ is resolvable, using the relations $\textbf{(4)}$ and $\textbf{(1)}$.

Let us show that the ambiguity $(x_{2i}x_{1j},x_{1j}d)$ is resolvable.

On the first hand, we get:
\begin{flushleft}
$\begin{aligned}
   &q^{-1}(x_{1i}x_{2j}d-C_{ij}d^2)\\
   \to &q^{-1}(-q^{-1}\sum_{k=1}^m(CD)_{jk}x_{1i}dx_{2k}-C_{ij}d^2)\\
   \to &q^{-1}(\sum_{k,l=1}^m(CD)_{jk}(C^{-1}D^{-1})_{li}dx_{1l}x_{2k}-C_{ij}d^2)
  \end{aligned}$
\end{flushleft}

and on the second hand: 
\begin{flushleft}
$\begin{aligned}
  &-q\sum_{k=1}^m(C^{-1}D^{-1})_{kj}x_{2i}dx_{1k} \\
  \to &\sum_{k,l=1}^m(C^{-1}D^{-1})_{kj}(CD)_{il}dx_{2l}x_{1k} \\
  \to &q^{-1}\sum_{k,l=1}^m(C^{-1}D^{-1})_{kj}(CD)_{il}d(x_{1l}x_{2k}-C_{lk}d)\\
  = &q^{-1}(\sum_{k,l=1}^m(C^{-1}D^{-1})_{kj}(CD)_{il}dx_{1l}x_{2k}-\sum_{k,l=1}^m(CD)_{il}C_{lk}(C^{-1}D^{-1})_{kj}d^2)\\ 
  = &q^{-1}(\sum_{k,l=1}^m(C^{-1}D^{-1})_{kj}(CD)_{il}dx_{1l}x_{2k}-C_{ij}d^2)\\
  = & q^{-1}(\sum_{k,l=1}^m(CD)_{jk}(C^{-1}D^{-1})_{li}dx_{1l}x_{2k}-C_{ij}d^2)   
  \end{aligned}$
\end{flushleft}

Let us show that the ambiguity $(x_{1m}x_{2v},x_{2v}d)$ is resolvable.

On the first hand we have:
\begin{flushleft}
$\begin{aligned}
   &(D_{mv})^{-1}\big(d^2-\sum_{(kl)<(mv)}D_{kl}x_{1k}x_{2l}d\big)\\
   \to &(D_{mv})^{-1}\big(d^2+q^{-1}\sum_{(kl)<(mv)}\sum_{j=1}^mD_{kl}(CD)_{lj}x_{1k}dx_{2j}\big)\\
   \to &(D_{mv})^{-1}\big(d^2-\sum_{(kl)<(mv)}\sum_{i,j=1}^m(C^{-1}D^{-1})_{ik}D_{kl}(CD)_{lj}dx_{1i}x_{2j}\big)\\
   = &(D_{mv})^{-1}\big(d^2-\sum_{i,j=1}^mD_{ij}dx_{1i}x_{2j}\big)+\sum_{i,j=1}^m(C^{-1}D^{-1})_{im}(CD)_{vj}dx_{1i}x_{2j}\\
   \to &\sum_{i,j=1}^m(C^{-1}D^{-1})_{im}(CD)_{vj}dx_{1i}x_{2j}
  \end{aligned}$
\end{flushleft}

because \[(D_{mv})^{-1}\big(d^2-\sum_{i,j=1}^mD_{ij}dx_{1i}x_{2j}\big)=(D_{mv})^{-1}\big(d^2-\sum_{(ij)<(mv)}D_{ij}dx_{1i}x_{2j}+D_{mv}dx_{1m}x_{2v}\big)\to 0\]
and on the second one, we have: 
\begin{flushleft}
$\begin{aligned}
   &-q^{-1}\sum_{j=1}^m(CD)_{vj}x_{1m}dx_{2j}\\
   \to &\sum_{i,j=1}^m(CD)_{vj}(C^{-1}D^{-1})_{im}dx_{1i}x_{2j}\\
  \end{aligned}$
\end{flushleft}

Let us show that the ambiguity $(x_{2m}x_{2v},x_{2v}d)$ is resolvable.

On the first hand, we have:
\begin{flushleft}
$\begin{aligned}
   &-(D_{mv})^{-1}\big(\sum_{(kl)<(mv)}D_{kl}x_{2k}x_{2l}d\big)\\
   \to &q^{-2}(D_{mv})^{-1}\big(\sum_{(kl)<(mv)}\sum_{i,j=1}^m(CD)_{ki}D_{kl}(CD)_{lj}dx_{2i}x_{2j}\big)\\
   = &(D_{mv})^{-1}\big(\sum_{(kl)<(mv)}\sum_{i,j=1}^m(C^{-1}D^{-1})_{ik}D_{kl}(CD)_{lj}dx_{2i}x_{2j}\big)\\
   = &(D_{mv})^{-1}\big(\sum_{i,j=1}^mD_{ij}dx_{2i}x_{2j}\big)-\sum_{i,j=1}^m(C^{-1}D^{-1})_{im}(CD)_{vj}dx_{2i}x_{2j}\\
   \to &-\sum_{i,j=1}^m(C^{-1}D^{-1})_{im}(CD)_{vj}dx_{2i}x_{2j}
  \end{aligned}$
\end{flushleft}

  on the second hand: 
\begin{flushleft}
$\begin{aligned}
   &-q^{-1}\sum_{j=1}^m(CD)_{vj}x_{2m}dx_{2j}\\
   \to &-q^{-2}\sum_{i,j=1}^m(CD)_{vj}(CD)_{mi}dx_{2i}x_{2m}\\
   = &-\sum_{i,j=1}^m(C^{-1}D^{-1})_{im}(CD)_{vj}dx_{2i}x_{2j}
  \end{aligned}$
\end{flushleft}

Let us show that the ambiguity $(x_{1m}x_{1v},x_{1v}d)$ is resolvable.

On the first hand, we get:
\begin{flushleft}
$\begin{aligned}
   &-(D_{mv})^{-1}\big(\sum_{(kl)<(mv)}D_{kl}x_{1k}x_{1l}d\big)\\
   \to &q(D_{mv})^{-1}\big(\sum_{(kl)<(mv)}D_{kl}(\sum_{j=1}^m(C^{-1}D^{-1})_{jl}x_{1k}dx_{1j})\big)\\
   \to &-q^2(D_{mv})^{-1}\big(\sum_{(kl)<(mv)}D_{kl}(\sum_{j=1}^m(C^{-1}D^{-1})_{jl}(\sum_{i=1}^m(C^{-1}D^{-1})_{ki}dx_{1i}x_{1j}))\big)\\
   = &-q^2(D_{mv})^{-1}\big(\sum_{(kl)<(mv)}\sum_{i,j=1}^m(C^{-1}D^{-1})_{ik}D_{kl}(C^{-1}D^{-1})_{jl}dx_{1i}x_{1j}\big)\\
   \end{aligned}$
$\begin{aligned}
   = &-(D_{mv})^{-1}\big(\sum_{(kl)<(mv)}\sum_{i,j=1}^m(C^{-1}D^{-1})_{ik}D_{kl}(CD)_{lj}dx_{1i}x_{1j}\big)\\
   = &-(D_{mv})^{-1}\sum_{i,j=1}^m D_{ij}dx_{1i}x_{1j}+\sum_{i,j=1}^m(C^{-1}D^{-1})_{im}(CD)_{vj}dx_{1i}x_{1j}\\
   \to &q^2\sum_{i,j=1}^m(C^{-1}D^{-1})_{im}(C^{-1}D^{-1})_{jv}dx_{1i}x_{1j}\\
  \end{aligned}$
\end{flushleft}

because $$\sum_{i,j=1}^m D_{ij}dx_{1i}x_{1j}= \sum_{(ij)<(mv)} D_{ij}dx_{1i}x_{1j}+D_{mv}dx_{1m}x_{1v}$$

on the second hand: 
\begin{flushleft}
$\begin{aligned}
   &-q\sum_{j=1}^m(C^{-1}D^{-1})_{jv}x_{1m}dx_{1j}\\
   \to &q^2\sum_{i,j=1}^m(C^{-1}D^{-1})_{jv}(C^{-1}D^{-1})_{im}dx_{1i}dx_{1j}\\
   = &q^2\sum_{i,j=1}^m(C^{-1}D^{-1})_{im}(C^{-1}D^{-1})_{jv}dx_{1i}dx_{1j}
  \end{aligned}$
\end{flushleft}

\end{proof}

Using this result, we can apply the diamond lemma and state:
\begin{cora}
The set of reduced monomials is a basis of $\mathcal{M}(A_q,A_q|C,D)$. In particular, the elements $x_{ij}$ are linearly independent and the algebra $\mathcal{M}(A_q,A_q|C,D)$ is non zero.
\end{cora}

In order to complete the proof of Lemma \ref{L 1}, we would like to add an inverse to $d$, and a good way to do this would be to localize $\mathcal{M}(A_q,A_q|C,D)$ by the multiplicative set $S=\{d^n, n\in \Bbb N\}$. By the presentation, we already have $\mathcal{M}(A_q,A_q|C,D)S=S\mathcal{M}(A_q,A_q|C,D)$, and we need to know that $d$ is not a zero divisor (see \cite{Dix}).

\begin{lemmaa}
$d\in \mathcal{M}(A_q,A_q|C,D)$ is not a zero divisor.
\end{lemmaa}

\begin{proof}
According the above lemma, the set of reduced monomials (denoted by $\Phi$) form a basis of $\mathcal{M}(A_q,A_q|C,D)$. A glance at the presentation show us that a reduced monomial is of the form $$d^ix, \ i\in \Bbb N, \ x \ \text{a ``good'' product of} \ x_{ij},$$ and the important thing to note is that if $M$ is a reduced monomial, so is $dM$. Finally, let $x=\sum_{M \in \Phi} \alpha_M M$ be an element of $\mathcal{M}(A_q,A_q|C,D)$ such that $dx=0$. Then $$dx=\sum_{M \in \Phi} \alpha_M dM=0 \Rightarrow \alpha_M=0, \ \forall \ \text{M reduced monomial},$$
then $x=0$.
\end{proof}

\begin{cora}
$\mathcal{G}(A_q,A_q|C,D)=\mathcal{M}(A_q,A_q|C,D)/S$ is non zero.
\end{cora}

\bibliographystyle{alpha}
\bibliography{bibliographie}

  \scshape 

\vglue0.3cm
\hglue0.02\linewidth\begin{minipage}{0.9\linewidth}
Colin Mrozinski\\
{Laboratoire de~Math\'ematiques (UMR 6620)}\\
Universit\'e Blaise Pascal\\
Complexe universitaire des C\'ezeaux\\
63177 Aubi\`ere Cedex, France \\
E-mail : \parbox[t]{0.45\linewidth}{\texttt{colin.mrozinski@math.univ-bpclermont.fr}} 
\end{minipage}

\end{document}